\documentclass[12pt, a4paper, reqno]{amsart}
\usepackage{amsthm,amsmath,amssymb,times}
\usepackage{graphicx,fullpage}
\usepackage[colorlinks,linkcolor=blue]{hyperref}
\usepackage{threeparttable,booktabs}
\usepackage{caption,mathrsfs}
\usepackage{subfigure,shadow,multirow}
\usepackage{cite,makecell,xcolor,colortbl}
\usepackage{url}
\hypersetup{plainpages=True, pdfstartview=FitH, bookmarksopen=true,
pdfpagemode=none, colorlinks=true,linkcolor=blue,citecolor=blue}

\colorlet{hili}{gray!20}
\def\le{\leqslant}
\def\ge{\geqslant}

\def\lpa#1{\bigl({#1}\bigr)}
\def\Lpa#1{\Bigl({#1}\Bigr)}

\def\llpa#1{\biggl({#1}\biggr)}
\def\dd#1{{\,\rm d}#1}
\def\ve{\varepsilon}

\def\eqtext#1{\quad\text{#1}\quad}

\newtheorem{thm}{Theorem}[section]

\makeatletter
\usepackage{tikz}
\newcommand*\circled[2][1.6]{\tikz[baseline=(char.base)]{
    \node[shape=circle, draw, inner sep=1pt, 
        minimum height={\f@size*#1},] (char) {\vphantom{WAH1g}#2};}}
\makeatother

\graphicspath{{./stir_formula_figs/}}

\def\and{\mbox{\quad and \quad}}

\begin{document}

\title[A curious identity in connection with saddle-point method]{A 
curious identity in connection with saddle-point method and 
Stirling's formula}

\let\origmaketitle\maketitle
\def\maketitle{
  \begingroup
  \def\uppercasenonmath##1{} 
  \let\MakeUppercase\relax 
  \origmaketitle
  \endgroup
}

\author{Hsien-Kuei Hwang}
\thanks{The research of the author was partially supported 
by Taiwan Ministry of Science and Technology under the Grant MOST 
108-2118-M-001-005-MY3.}

\address{Institute of Statistical Science, Academia Sinica, Taipei,
115, Taiwan}
\date{\today}
\email{\textbf{hkhwang@stat.sinica.edu.tw}}

\maketitle

\begin{abstract}
	
We prove the curious identity in the sense of formal power series:
\[
	\int_{-\infty}^{\infty}[y^m]
	\exp\left(-\frac{t^2}2
	+\sum_{j\ge3}\frac{(it)^j}{j!}\, 
	y^{j-2}\right)\mathrm{d} t
	= \int_{-\infty}^{\infty}[y^m]
    \exp\left(-\frac{t^2}2+
	\sum_{j\ge3}\frac{(it)^j}{j}\, 
	y^{j-2}\right)\mathrm{d} t,
\]
for $m=0,1,\dots$, where $[y^m]f(y)$ denotes the coefficient of 
$y^m$ in the Taylor expansion of $f$. The generality of this identity 
from the perspective of saddle-point method is also examined. 

\end{abstract}

\section{Introduction}

The following unusual identity was discovered through different
manipulations of the saddle-point method in order to derive
Stirling's formula, which has a huge literature since de Moivre's
and Stirling's pioneering analysis in the early eighteenth century;
see for example the survey \cite{Borwein2018} (and the references
therein) and the book \cite{Flajolet2009} for five different
analytic proofs. While the identity can be deduced from known
expansions for $n!$ (e.g., \cite{Brassesco2011,Paris2014}), the 
formulation, as well as the proof given here, is of independent 
interest \emph{per se}. Denote by $[y^m]f(y)$ the coefficient of 
$y^m$ in the Taylor expansion of $f$.
\begin{thm} Let \label{T:am-bm}
\begin{align}\label{E:am}
	c_m := \frac{1}{\sqrt{2\pi}}
	\int_{-\infty}^{\infty}[y^m]
	\exp\llpa{-\frac{t^2}2
	+\sum_{j\ge3}\frac{(it)^j}
	{\colorbox{hili}{$j!$}}\, y^{j-2}}\dd t,
\end{align}
and 
\begin{align}\label{E:bm}
	d_m := \frac{1}{\sqrt{2\pi}}
	\int_{-\infty}^{\infty}
    [y^m]\exp\llpa{-\frac{t^2}2+
	\sum_{j\ge3}\frac{(it)^j}
	{\colorbox{hili}{$j$}}\, y^{j-2}}\dd t.
\end{align}
Then 
\begin{align}\label{E:id}
	c_m=d_m\qquad (m=0,1,\dots).
\end{align}
\end{thm}

When $m$ is odd, $c_m=d_m=0$ because the coefficient of $y^m$ 
contains only odd powers of $t$. When $m=2l$ is even, the identity 
\eqref{E:id} can be written explicitly as follows:
\begin{align*}
	c_{2l}&=\sum_{1\le h\le 2l}
	\frac{(-1)^{l+h}(2l+2h)!}{(l+h)!2^{l+h}}
	\sum_{\substack{j_1+2j_2+\cdots+2lj_{2l}=2l
	\\ j_1+\cdots+j_{2l}=h}}
	\frac{1}
	{j_1!\cdots j_{2l}!
	\cdot \colorbox{hili}{$3!^{j_1}4!^{j_2}
	\cdots (2l+2)!^{j_{2l}}$}}\\
	&= \sum_{1\le h\le 2l}
	\frac{(-1)^{l+h}(2l+2h)!}{(l+h)!2^{l+h}}
	\sum_{\substack{j_1+2j_2+\cdots+2lj_{2l}=2l
	\\ j_1+\cdots+j_{2l}=h}}
	\frac{1}
	{j_1!\cdots j_{2l}!
	\cdot \colorbox{hili}{$3^{j_1}4^{j_2}
	\cdots (2l+2)^{j_{2l}}$}}\\
	&= d_{2l}.
\end{align*}
In particular, 
\[
	\{c_{2l}\}_{l\ge0}
	= \left\{1,-\frac1{12},\frac1{288}, \frac{139}{51840},
	-\frac{571}{2488320},-\frac{163879}{209018880},\cdots\right\},
\]
which are modulo sign the coefficients appearing in the asymptotic 
expansion of Stirling's formula; see \cite[\S 1.18]{Erdelyi1953} or 
\cite[\href{https://oeis.org/A001163}{A001163}, 
\href{https://oeis.org/A001164}{A001164}]{OEIS2022}:
\begin{align}\label{E:stir-cm0}	
	\frac1{n!}\sim \frac{e^nn^{-n-\frac12}}{\sqrt{2\pi}}
	\sum_{m\ge0}c_{2m}n^{-m},\eqtext{or}
	n!\sim \sqrt{2\pi}e^{-n}n^{n+\frac12}
	\sum_{m\ge0}(-1)^m c_{2m}n^{-m}.
\end{align}
These Stirling coefficients have been extensively studied in the 
literature; see, e.g., \cite[\S 8.2]{Dingle1973}, and 
\cite{Boyd1994,Mortici2010,Paris2014,Nemes2015}, and the 
references cited there.

On the other hand, the integral in \eqref{E:am} without the 
coefficient-extraction operator $[y^m]$ is divergent for 
$y\in\mathbb{R}$ due to periodicity:
\[
	\int_{-\infty}^{\infty}
	\exp\llpa{-\frac{t^2}2
	+\sum_{j\ge3}\frac{(it)^j}{j!}\, y^{j-2}}\dd t
	=\int_{-\infty}^\infty
	\exp\llpa{\frac{e^{ity}-1-ity}{y^2}}\dd t,
\]
while that in \eqref{E:bm} is absolutely convergent for real 
$|y|<1$:
\[
	\int_{-\infty}^{\infty}
	\exp\llpa{-\frac{t^2}2+
    \sum_{j\ge3}\frac{(it)^j}{j}\, y^{j-2}}\dd t
	= \int_{-\infty}^{\infty}
	\lpa{1-ity}^{-y^{-2}}e^{-it/y}\dd t.
\]	

\begin{proof}[Proof of Theorem~\ref{T:am-bm}]
For convenience, we write
\[
	f_n \thickapprox g_n
	\eqtext{when}f_n = g_n + O\lpa{e^{-\ve n}},
\]
for some generic $\ve>0$ whose value is immaterial.

\medskip

\paragraph{\textbf{The standard asymptotic expansion}}
We begin with the Cauchy integral representation for $n!^{-1}$:
\[
	\frac1{n!}
	= \frac1{2\pi i}\oint_{|z|=n} z^{-n-1}e^z \dd z,
\]
where the integration contour is the circle with radius $|z|=n$. 
The standard application of the saddle-point method (see 
\cite[p.~555]{Flajolet2009}) proceeds by first making the change of 
variables $z\mapsto ne^u$, giving 
\begin{align}\label{E:sd1}
	\frac{e^{-n}n^{n}}{n!}
	&= \frac1{2\pi i}
	\int_{-\pi i}^{\pi i}e^{n(e^u-1-u)} \dd u
	\thickapprox  \frac1{2\pi i}\int_{-\ve i}^{\ve i}
	e^{n(e^u-1-u)}\dd u.
\end{align}
Now by the change of variables $u=\frac{it}{\sqrt{n}}$, we have
\begin{align*}
	\frac1{2\pi i}\int_{-\ve i}^{\ve i}
	e^{n(e^u-1-u)}\dd u
	&= \frac1{2\pi \sqrt{n}}
	\int_{-\ve \sqrt{n}}^{\ve \sqrt{n}}
	\exp\llpa{-\frac{t^2}2
	+\sum_{j\ge3}\frac{(it)^j}{j!}
	\, n^{-\frac12j+1}}\dd t.
\end{align*}
If we choose $\ve=\ve_n=n^{-\frac25}$, say, then
$n\ve_n^2\to\infty$ and $n\ve_n^j\to0$ for $j\ge3$, so that the
series on the right-hand side is small on the integration path; we
can then expand the exponential of this series in decreasing powers
of $n$, and then extending the integration limits to infinity,
yielding the expansion \eqref{E:stir-cm0} with $c_{2m}$ expressed
in the formal power series form \eqref{E:am}. See \cite[Ex.~VIII.3;
p.~555 \emph{et seq.}]{Flajolet2009} for technical details.

On the other hand, a more effective means of computing $c_{2m}$ is to
make first the change of variables $e^u -1-u=\frac12v^2$ in the
rightmost integral in \eqref{E:sd1}, where $u=u(v)$ is positive when
$v$ is and is analytic in $|v|\le \ve$; see \cite[\S~3.6.3]{Temme1996}. Then
\[
	\frac{e^{-n}n^{n}}{n!}
    \thickapprox  
	\frac1{2\pi i}\int_{-\ve i}^{\ve i}
	e^{\frac12 nv^2}g(v)\dd v
	= \frac1{2\pi}\int_{-\ve}^{\ve}
	e^{-\frac12 nt^2}g(it)\dd t,
\]
where $g(v) := \frac{\dd u}{\dd v}$ is analytic in $|v|\le\ve$. By 
Lagrange inversion formula (see \cite[p.~732]{Flajolet2009}),
\begin{align}\label{E:gm}
	g_m := [v^m]g(v) 
	= [t^m]\llpa{\frac{\frac12t^2}{e^t-1-t}}^{\frac12(m+1)}
	\qquad(m=0,1,\dots).
\end{align}
Then a direct application of Watson's Lemma (see 
\cite[\S 1.5]{Wong2001}) gives the asymptotic expansion
\begin{align}\label{E:stir-cm}
	\frac1{n!}
	\thickapprox \frac{e^{n}n^{-n}}{2\pi}
	\sum_{m\ge0}g_{2m}\int_{-\infty}^{\infty}
	e^{-\frac12nt^2}(it)^{2m}\dd t
	\thickapprox \frac{e^{n}n^{-n}}{\sqrt{2\pi n}}
	\sum_{m\ge0}\bar{g}_{2m} n^{-m},
\end{align}
where
\[
	\{\bar{g}_{2m}\}_{m\ge0} 
	:= \Bigl\{g_{2m}\frac{(-1)^m(2m)!}{m!2^m}\Bigr\}_{m\ge0}
	= \Bigl\{1, -\frac1{12}, \frac{1}{288}, 
	\frac{139}{51840}, -\frac{571}{2488320}, 
	\cdots \Bigr\}.
\]
We then obtain the relation
\begin{align}\label{E:am-g2m}
	c_{2m} = \bar{g}_{2m}
	= g_{2m}\frac{(-1)^m(2m)!}{m!2^m}.
\end{align}

\medskip

\paragraph{\textbf{Second asymptotic expansion}}
It is well known that $n!^{-1}$ has the alternative Laplace integral 
representation (see \cite[p.~246]{Whittaker1927}):
\[
	\frac1{n!}
	= \frac1{2\pi i}\int_{R-i\infty}^{R+i\infty} 
	z^{-n-1}e^z \dd z\qquad(R>0),
\]
so that, by the change of variables $z=R(1+x)$, where $R=n+1$,
\begin{align*}
	\frac{e^{-n-1}(n+1)^{n}}{n!}
	&= \frac1{2\pi i}\int_{-i\infty}^{i\infty}
	e^{-(n+1)(\log(1+x)-x)} \dd x\\
	&\thickapprox \frac1{2\pi i}\int_{-\ve i}^{\ve i}
	e^{-(n+1)(\log(1+x)-x)} \dd x.
\end{align*}
Now by the change of variables $x=\frac{it}{\sqrt{n+1}}$, we have
\begin{align*}
	&\frac1{2\pi i}\int_{-\ve i}^{\ve i}
	e^{-(n+1)(\log(1+x)-x)} \dd x\\
	&\qquad= \frac1{2\pi \sqrt{n+1}}
	\int_{-\ve \sqrt{n+1}}^{\ve \sqrt{n+1}}
	\exp\llpa{-\frac{t^2}2
	+\sum_{j\ge3}\frac{(it)^j}{j}
	\, (n+1)^{-\frac12j+1}}\dd t.
\end{align*}
By a similar procedure described above, we then deduce the asymptotic 
expansion
\begin{align}\label{E:stir-dm}
	\frac1{n!}
	\sim \frac{e^{n+1}(n+1)^{-n-\frac12}}
	{\sqrt{2\pi}}\sum_{m\ge0}d_{2m} (n+1)^{-m},
\end{align}
where $d_m$ is given in \eqref{E:bm}; compare \eqref{E:stir-cm}.

On the other hand, by the change of variables 
$\log(1+x)-x=-\frac12y^2$ ($y>0$ when $x>0$), we have 
\begin{align*}
	\frac{e^{-n-1}(n+1)^{n}}{n!}
	&\thickapprox  \frac1{2\pi i}\int_{-\ve i}^{\ve i}
	e^{\frac12(n+1)y^2}h(y) \dd y
	= \frac1{2\pi}\int_{-\ve}^{\ve}
	e^{-\frac12(n+1)t^2}h(it) \dd t,
\end{align*}
where $h(y) = \frac{\dd x}{\dd y}$ is analytic in $|y|\le\ve$. 
Again, by Lagrange inversion formula,
\begin{align}\label{E:hm}
	h_m := [y^m]h(y)
	= [y^m]\llpa{\frac{\frac12y^2}
	{y-\log(1+y)}}^{\frac12(m+1)}
	\qquad(m=0,1,\dots).
\end{align}
Although the definition of $h_m$ looks very different from that of 
$g_m$ (see \eqref{E:gm}), their numerical values coincide except 
for $m=1$:
\begin{center}
\begin{tabular}{c|c>{\columncolor[gray]{0.7}}ccccccccc}
$m$ & $0$ & $1$ & $2$ & $3$ & $4$ & $5$ & $6$ 
& $7$ & $8$ & $9$\\ \hline	
$g_m$ 
 & \multirow{2}{*}{$1$} 
 & $-\frac13$
 & \multirow{2}{*}{$\frac1{12}$} 
 & \multirow{2}{*}{$-\frac 2{135}$} 
 & \multirow{2}{*}{$\frac1{864}$} 
 & \multirow{2}{*}{$\frac1{2835}$} 
 & \multirow{2}{*}{$-\frac{139}{777600}$}
 & \multirow{2}{*}{$\frac1{25515}$} 
 & \multirow{2}{*}{$-\frac{571}{261273600}$} 
 & \multirow{2}{*}{$-\frac{281}{151559100}$}\\
 $h_m$ & &  $\frac23$ 
\end{tabular}	
\end{center}
We thus deduce the relation 
\[
    d_{2m} := h_{2m}\frac{(-1)^m(2m)!}{m!2^m},
\]
which is easily computable by \eqref{E:hm}.

\medskip

\paragraph{\textbf{Equality of the two expansions}}
We next prove that 
\begin{align}\label{E:gh}
	g_m = h_m \qquad(m\ge0; m\ne1),
\end{align}
where $g_m$ and $h_m$ are defined in \eqref{E:gm} and \eqref{E:hm}, 
respectively. Note that for $m\ge0$
\begin{align}\label{E:phi-s}
	g_m = [s^m]\frac{\varphi(s)^{m+1}}{1+s},
	\eqtext{with}
	\varphi(s) := \llpa{\frac{s^2}{2(s-\log(1+s))}}^{\frac12},
\end{align}
by a direct change of variables $s=e^t-1$. Thus we show that
\[
	[s^m]\frac{\varphi(s)^{m+1}}{1+s}
	=[s^m]\varphi(s)^{m+1}
\]
for $m\ne1$, or, equivalently, 
\begin{align}\label{E:zero}
	[s^{m-1}]\frac{\varphi(s)^{m+1}}{1+s} = 0
	\qquad(m\ge0, m\ne1).
\end{align}
Since $m=0,1$ are easily checked, we assume $m\ge2$. By the 
relation 
\[
	\frac{s\varphi'(s)}{\varphi(s)}
	= 1-\frac{\varphi(s)^2}{1+s},
\]
we have 
\begin{align*}
	[s^{m-1}]\frac{\varphi(s)^{m+1}}{1+s}
	&= [s^{m-1}]\lpa{\varphi(s)^{m-1}-
	s\varphi(s)^{m-2}\varphi'(s)}\\
	&= h_{m-1} - [s^{m-2}]\varphi(s)^{m-2}\varphi'(s).
\end{align*}
Now
\begin{align*}
	[s^{m-2}]\varphi(s)^{m-2}\varphi'(s)
	&= [s^{m-2}]\frac{\dd{}}{\dd s}\frac{\varphi(s)^{m-1}}{m-1}
	= [s^{m-1}]\varphi(s)^{m-1}=h_{m-1},
\end{align*}
which proves \eqref{E:zero}, and in turn \eqref{E:gh}. 
Consequently, $c_m=d_m$ for $m\ge0$, implying \eqref{E:id}.
\end{proof}

It is known that (see \cite[\S 8.2]{Dingle1973} and \cite{Boyd1994})
\begin{align*}
	g_{2m}
	\sim (-1)^{l+1}\sqrt{2\pi}(4\pi)^{-2l}
	\times \begin{cases}
		\frac1{24}l^{-\frac32}, &\text{if }m=2l;\\
		2l^{-\frac12}, & \text{if }m=2l-1.
	\end{cases}
\end{align*}
This type of asymptotic behaviors is unusual for functions of 
Lagrangean type; see \cite{Hwang2018} or \S~\ref{S:lag}. 

\section{Asymptotic expansions by saddle-point method}


Quoted from \cite[p.~551]{Flajolet2009}
\begin{quote}\small
\[
	\text{\emph{Saddle-point method 
		$=$ Choice of contour $+$ Laplace’s method.}}
\]
\emph{Similar to its real-variable counterpart, the saddle-point
method is a general strategy rather than a completely deterministic
algorithm, since many choices are left open in the implementation of
the method concerning details of the contour and choices of its
splitting into pieces.}
\end{quote}

\subsection{$z=re^{i\theta}$ or $z=R(1+it)$?}

The two uses above (with $z=re^{i\theta}$ or $z=R(1+it)$) of the 
saddle-point method for coefficient integrals of the form 
\[
	a_n 
	:= \frac1{2\pi i}\int_{\mathscr{C}} z^{-n-1} f(z) \dd z,
\]
for some contour $\mathscr{C}$ are standard in the combinatorial
literature and are reminiscent of the difference between moments and 
factorial moments in probability; the corresponding saddle-point 
equations are given by
\begin{align}\label{E:n-vs-np1}
	\frac{rf'(r)}{f(r)} = n,
	\eqtext{and}
	\frac{Rf'(R)}{f(R)} = n+1,
\end{align}
respectively. Often the question is which one to choose and is better
(numerically or in some other sense)? For example, take $f(z) =
e^{e^z-1}$ (Bell numbers
\cite[\href{https://oeis.org/A000110}{A000110}]{OEIS2022}); then we
found both uses in the literature:

\medskip
\begin{center}
	\begin{tabular}{c|c}
		$re^r=n$ & $Re^R=n+1$ \\ \hline
		\cite{Moser1955,Szekeres1957}
		\cite[Ex.~12.2]{Odlyzko1995}
		\cite[\S~5.8]{Sachkov1996}
		& \cite[\S~6.2]{deBruijn1981} 
		\cite[pp.~560--562]{Flajolet2009}
		\cite[pp.~422--423]{Knuth2011}
	\end{tabular}
\end{center}
\medskip

\noindent
In particular, Knuth \cite[pp.~422--423]{Knuth2011} considers first 
$a_{n-1}$ and then changed $n-1$ to $n$ after deriving the 
corresponding asymptotic approximation. 

The question of whether to use $r$ or $R$ in \eqref{E:n-vs-np1} is
partly answered in \cite[p.~555, footnote]{Flajolet2009}: ``\emph{the
choice being often suggested by computational convenience}.'' Odlyzko
in \cite[p.~1184]{Odlyzko1995} also commented that the use of $r$ is
slightly preferred because the manipulation of the other version is 
less elegant. 

Apart from computational convenience, the numerical advantages of the
expansion \eqref{E:stir-dm} over \eqref{E:stir-cm} are visible
because they have the same sequence of coefficients and $(n+1)^{-m}$
is always smaller than $n^{-m}$; see also \cite{Corless2019} for
Stirling's original expansion in decreasing powers of $n+\frac12$.
Although the numerical difference is minor for most practical uses,
the same question can naturally be raised more generally for
functions $f$ whose Taylor coefficients are amenable to saddle-point
method (for example, exponential of Hayman admissible functions; see
\cite{Odlyzko1985}). Indeed, such a numerical difference was already
observed in the 1960s by Harris and Schoenfeld in their study of
idempotent elements in symmetric semigroups \cite{Harris1967} where
$f(z) = e^{ze^z}$. Based on numerical calculations, they found that
the saddle-point approximation
\begin{align}
	\frac{a_n}{n!} &:= [z^n]e^{ze^z}\nonumber \\
	&\sim \frac{R^{-n}e^{Re^R}}
	{\sqrt{2\pi Re^R(R^2+3R+1)}}
	\eqtext{with $R>0$ solving $R(R+1)e^R=n+1$,} \label{E:Jn}
\end{align}
is ``\emph{considerably better than the approximation}''
\begin{align}\label{E:In}
	\frac{a_n}{n!} 
	\sim \frac{r^{-n}e^{re^r}}
	{\sqrt{2\pi re^r(r^2+3r+1)}}
	\eqtext{with $r>0$ solving $r(r+1)e^r=n$.}
\end{align}
Surprisingly, this is the only paper we found where such a numerical 
comparison between the two versions of saddle-point method was 
made. 

In the same paper \cite{Harris1967}, Harris and Schoenfeld also
argued that the reason that \eqref{E:Jn} outperforms \eqref{E:In} is
(we change their notations to ours) ``\emph{because the derivation of
\eqref{E:Jn} uses a contour which passes through the saddle point of
a certain integral for $a_n/n!$. However, Hayman's proof of the
formula yielding \eqref{E:In} employs a contour passing through
$r=r(n) = R(n-1)$ and it therefore misses the saddle point at $R$ by
$R(n) - R(n-1) \sim 1/n$.}"

However, such a comparison is not quite right. In fact, the use of
$r$ or $R$ in each case, after the change of variables, is optimally
guided by the saddle-point principle, and thus different choices of
integration contour will result in different expansions with
different asymptotic scales. As we will see below in
\S~\ref{S:examples}, while the dominant term in \eqref{E:Jn} is
better than that in \eqref{E:In} under the absolute difference
measure, the use of more terms in the corresponding asymptotic
expansions may change the scenario, and which expansion is
numerically more precise depends on the number of terms used.

In this section, we first consider the two versions of the
saddle-point method for general $f$, and then give succinct
expressions for the coefficients in the corresponding asymptotic
expansions. Then we discuss some examples, highlighting briefly their 
numerical differences. 

\subsection{Two asymptotic expansions by saddle-point method}

Consider
\[
	a_n := [z^n]e^{\phi(z)}
	= \frac1{2\pi i}\oint_{|z|=r} z^{-n-1}e^{\phi(z)}\dd z,
\]
where $\phi$ is analytic at zero. For simplicity we will simply
assume that asymptotic expansions can be derived by the saddle-point
method, instead of formulating general theorems for asymptotic
expansions whose conditions are often very messy to the extent that
in many practical cases, their justifications are not much different
from working out the saddle-point method from scratch; see e.g.
\cite{Harris1968, Moser1957, Odlyzko1985, Wyman1959}. In this way, we
focus on the formal aspects of the coefficients and the differences
between the two expansions, referring all technical details or
justification to standard references such as
\cite{Harris1968,Moser1957,Wyman1959} or
\cite[\S~VIII.5]{Flajolet2009}. Then we have the following two
different asymptotic expansions.
\begin{itemize}
\item The change of variables $z\mapsto re^{i\theta}$ (integration 
on a circle): let 
\begin{align*}
	\kappa_j(r) 
	&:= j![s^j]\lpa{-ns+\phi(re^s)}
	\qquad(j=1,2,\dots),
\end{align*} 
where $\kappa_1(r)=0$ or $r\phi'(r) = n$. Also $\kappa_2(r) = 
r\phi'(r)+r^2\phi''(r)$. Then under proper conditions
\begin{align*}
	a_n 
	\sim \frac{r^{-n}e^{\phi(r)}}{\sqrt{2\pi\kappa_2(r)}}
	\sum_{m\ge0}c_m(r) \kappa_2(r)^{-m},
\end{align*}
where
\begin{align*}
	c_m(r) := \frac1{\sqrt{2\pi}}
	\int_{-\infty}^{\infty}e^{-\frac12 t^2}
	[y^{m}]\exp\llpa{\sum_{j\ge3}
	\frac{\kappa_j(r)}{j!\kappa_2(r)}(it)^j y^{\frac12j-1}}\dd t.
\end{align*}


Alternatively, by the same analysis as above for Stirling's formula,
we also have
\begin{align}\label{E:cm-gen}
	c_m(r) = g_{2m}(r) 
	\frac{(-1)^m(2m)!}{m!2^m}, 
\end{align}
where
\[
	g_m(r) = [v^m]\llpa{\frac{\frac12\kappa_2(r)v^2}
	{\phi(re^v)-\phi(r)-r\phi'(r)v}}^{\frac12(m+1)}.
\]

In particular (with $\kappa_j=\kappa_j(r)$),
\begin{align*}
	c_1(r) &= \frac{3\kappa_2\kappa_4-5\kappa_3^2}
	{24\kappa_2^2},\\
	c_2(r) &= -\frac{24\kappa_2^3\kappa_6
	-168\kappa_2^2\kappa_3\kappa_5
	-105\kappa_2^2\kappa_4^2
	+630\kappa_2\kappa_3^2\kappa_4
	-385\kappa_3^4}{1152\kappa_2^4}.
\end{align*}

\item The change of variables $z\mapsto R(1+it)$ (integration on a 
vertical line): let 
\begin{align*}
	\lambda_j(R)
	&= j![s^j]\lpa{-(n+1)\log(1+s)
	+\phi(R(1+s))}\\
	&= (-1)^j(j-1)!R\phi'(R)+j![s^j]\phi(R(1+s)),
\end{align*}
where $R\phi'(R)= n+1$. Note that $\lambda_2(t)=\kappa_2(t)$. 
Then under proper conditions 
\begin{align*}
	a_n 
	\sim \frac{R^{-n}e^{\phi(R)}}{\sqrt{2\pi\lambda_2(R)}}
	\sum_{m\ge0}d_m(R) \lambda_2(R)^{-m},
\end{align*}
where
\begin{align*}
	d_m(R) := \frac1{\sqrt{2\pi}}
	\int_{-\infty}^{\infty}e^{-\frac12 t^2}
	[y^m]\exp\llpa{\sum_{j\ge3}
	\frac{\lambda_j(r)}
	{j!\lambda_2(r)}(it)^j y^{\frac12j-1}}\dd t.
\end{align*}

Alternatively, 
we also have
\begin{align}\label{E:dm-gen}
	d_m(R)
	= h_{2m}(R) \frac{(-1)^m(2m)!}{m!2^m}, 
\end{align}
where
\[
	h_m(R) = [y^m]\llpa{\frac{\frac12\lambda_2(R)y^2}
	{\phi(R(1+y))-\phi(R)-R\phi'(R)\log(1+y)}}^{\frac12(m+1)}.
\]

In particular, (with $\lambda_j=\lambda_j(R)$)
\begin{align*}
	d_1(R) &= \frac{3\lambda_2\lambda_4-5\lambda_3^2}
	{24\lambda_2^2},\\
	d_2(R) &= -\frac{24\lambda_2^3\lambda_6
	-168\lambda_2^2\lambda_3\lambda_5
	-105\lambda_2^2\lambda_4^2
	+630\lambda_2\lambda_3^2\lambda_4
	-385\lambda_3^4}{1152\lambda_2^4},
\end{align*}
the expressions differing from $c_1(r)$ and $c_2(r)$ by replacing 
all $\kappa_j(r)$ by $\lambda_j(R)$.
\end{itemize}

%
%
%

\subsection{Examples.}\label{S:examples}

We begin with Harris and Schoenfeld's example $\phi(z)=ze^z$
\cite{Harris1967} and let $a_n := n![z^n]e^{ze^z}$, the number of
idempotent mappings from a set of $n$ elements into itself; see also
\cite[\href{https://oeis.org/A000248}{A000248}]{OEIS2022}. Asymptotic
expansions by saddle-point method can be justified either by checking
the Harris-Schoenfeld admissibility conditions as in
\cite{Harris1967} or by showing that $ze^z$ is a Hayman admissible
function (see \cite{Hayman1956, Odlyzko1985, Flajolet2009}). We then
compute the absolute differences between the true values and the two
asymptotic expansions with varying number of terms:
\begin{equation}\label{E:Delta}
\begin{split}
	\Delta_{n,M}^{(c)}
	&:= \frac{n^{M+1}}{(\log n)^{M+1}}
	\left|\frac{a_n}{\displaystyle
	\frac{r^{-n}e^{re^r}}
	{\sqrt{2\pi\kappa_2(r)}}}-
	\sum_{0\le m\le M}c_m(r) \kappa_2(r)^{-m}\right|,\\
	\Delta_{n,M}^{(v)}
	&:= \frac{n^{M+1}}{(\log n)^{M+1}}
	\left|\frac{a_n}{\displaystyle
	\frac{R^{-n}e^{Re^R}}
	{\sqrt{2\pi\kappa_2(R)}}}-
	\sum_{0\le m\le M}d_m(R) \kappa_2(R)^{-m}\right|,
\end{split}
\end{equation}
where $r>0$ solves $r(r+1)e^r=n$, $R>0$ solves $R(r+1)e^R=n+1$, and 
$\kappa_2$ and the coefficients $c_m$ and $d_m$ can be computed by 
\eqref{E:cm-gen} and \eqref{E:dm-gen}, respectively, with $\phi(z) = 
ze^z$. Note that $g(2m)\kappa_2(r)^{-m}$ grows in the order 
$n^{-m}(\log n)^m$ for $m=0,1,\dots$.
\begin{figure}[!ht]
	\begin{center}
	\begin{tabular}{ccccc}
		\includegraphics[height=2.5cm]{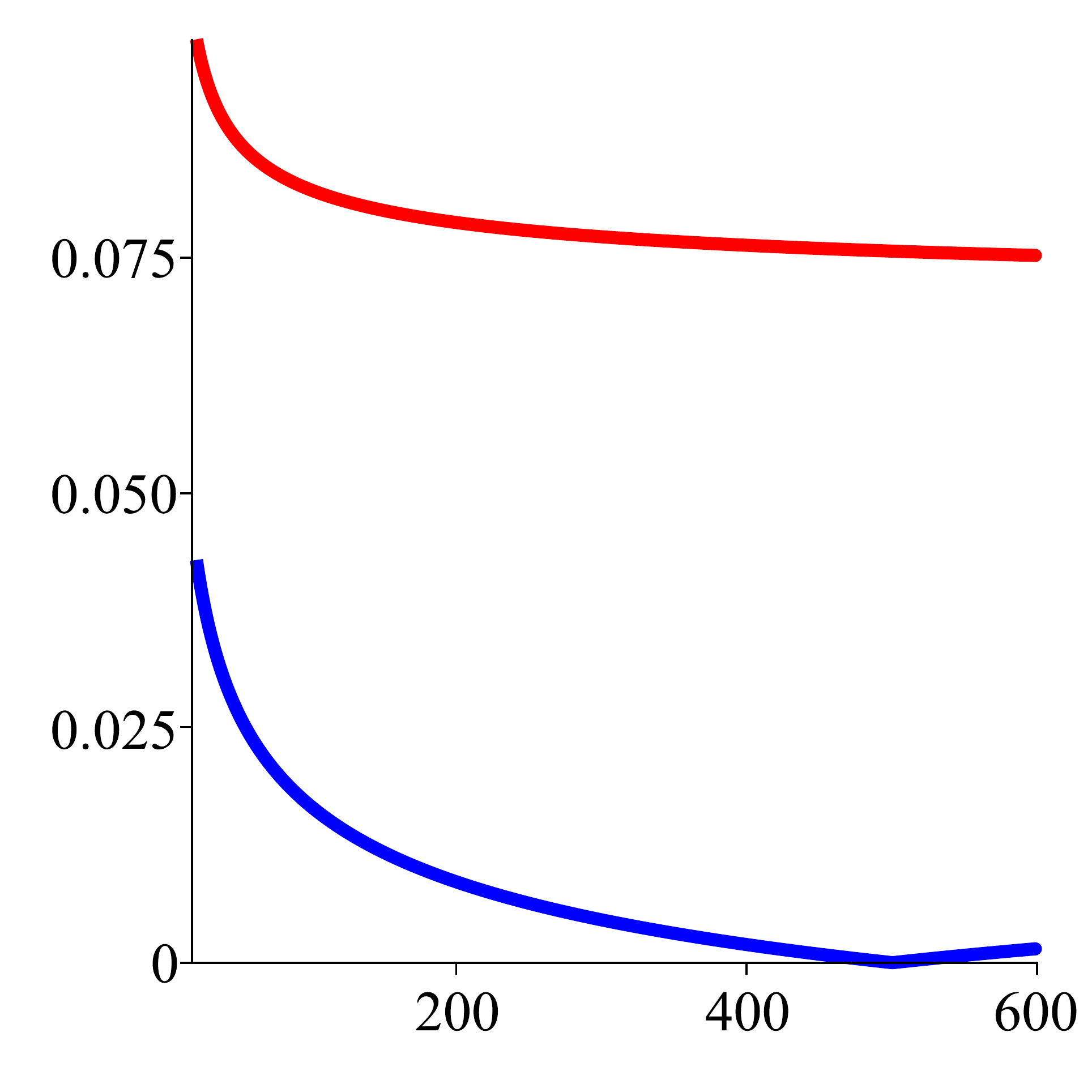}
		& \includegraphics[height=2.5cm]{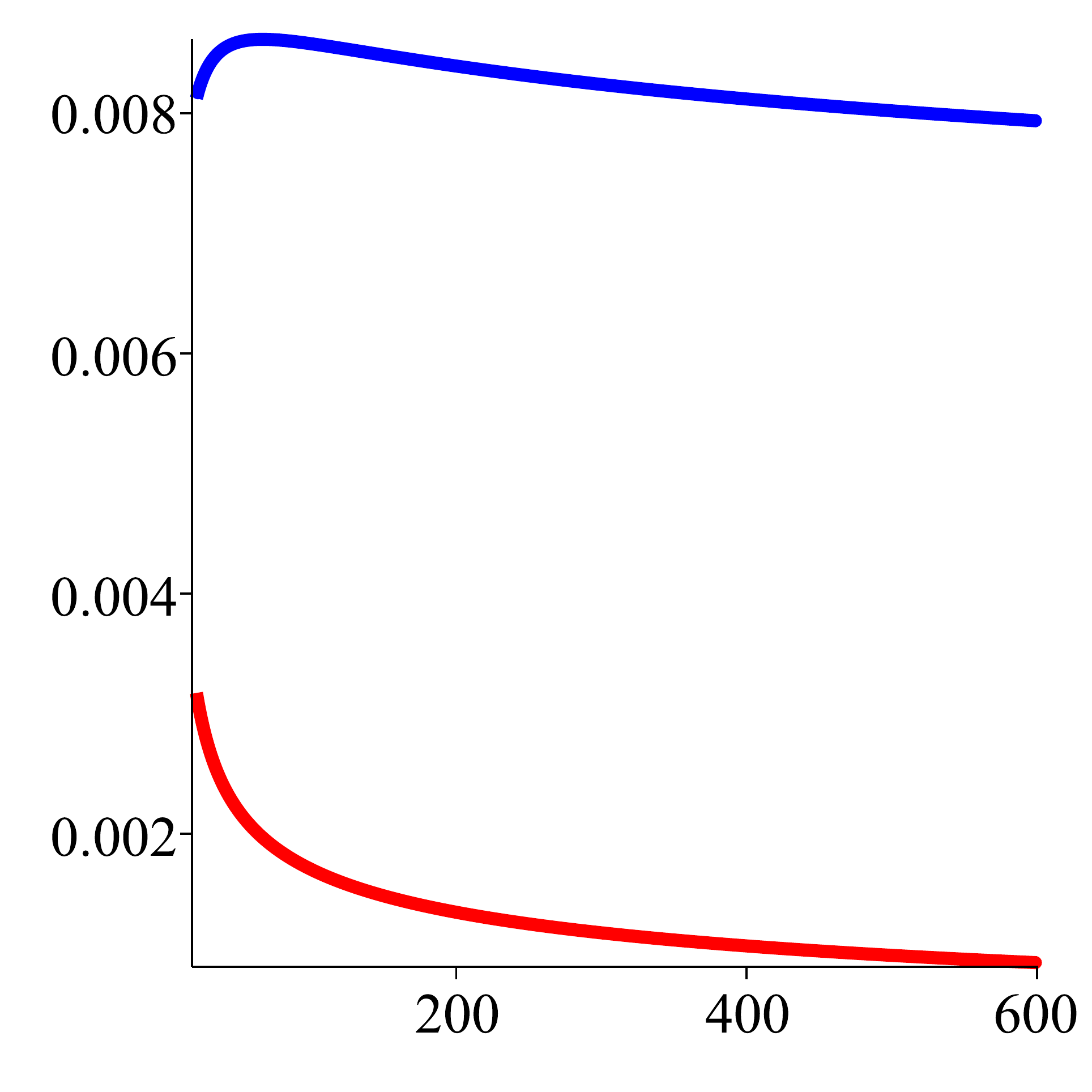}
		& \includegraphics[height=2.5cm]{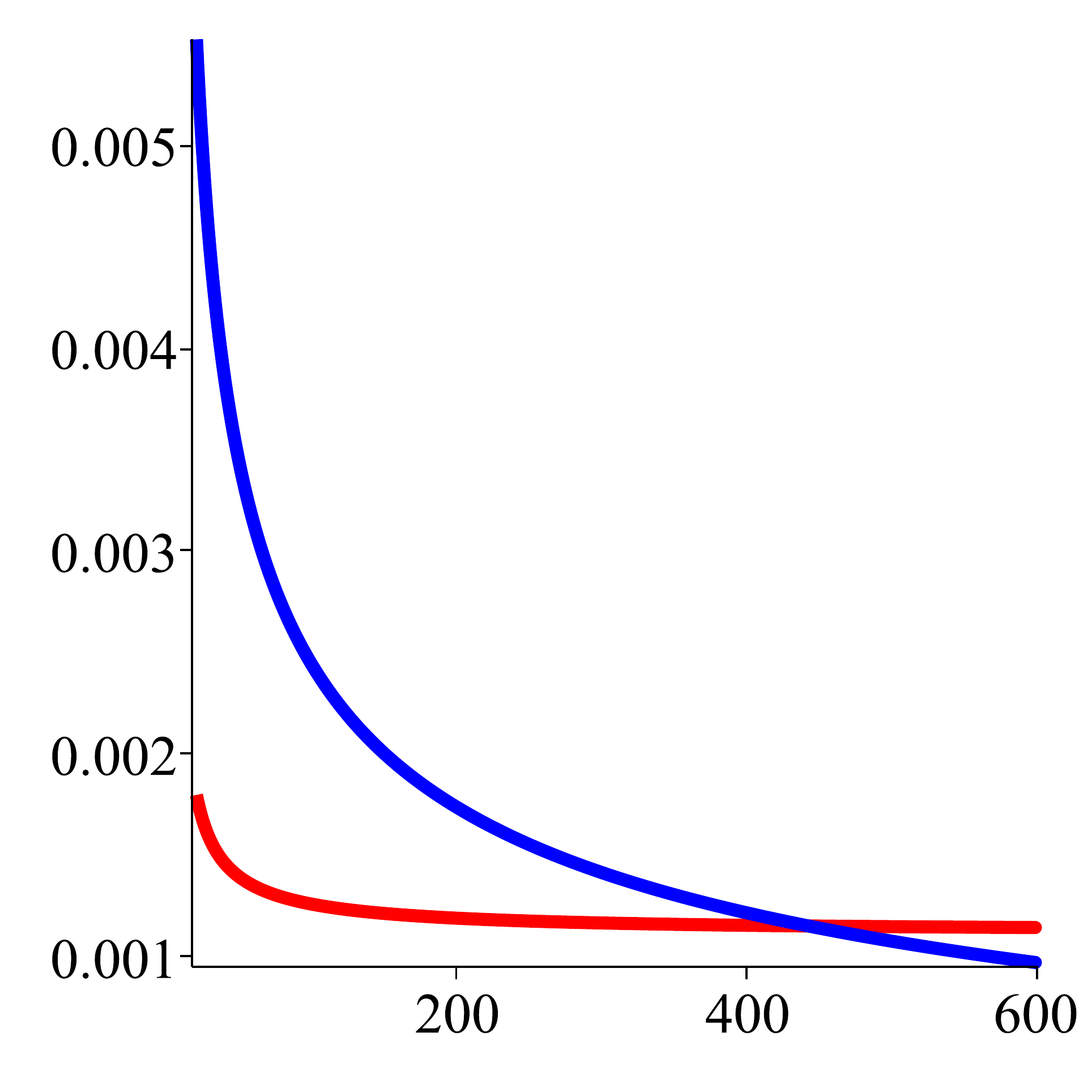}
		& \includegraphics[height=2.5cm]{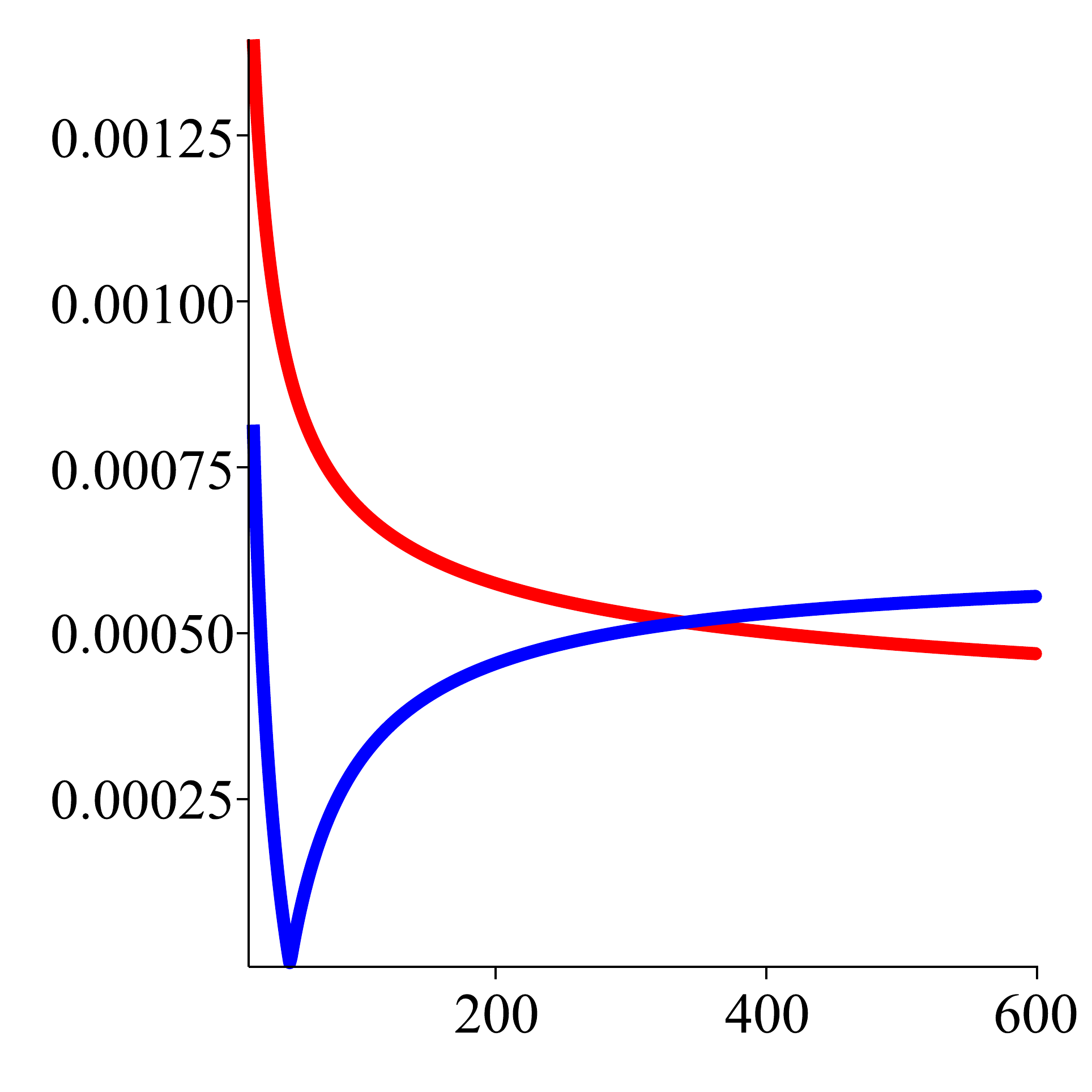}
		& \includegraphics[height=2.5cm]{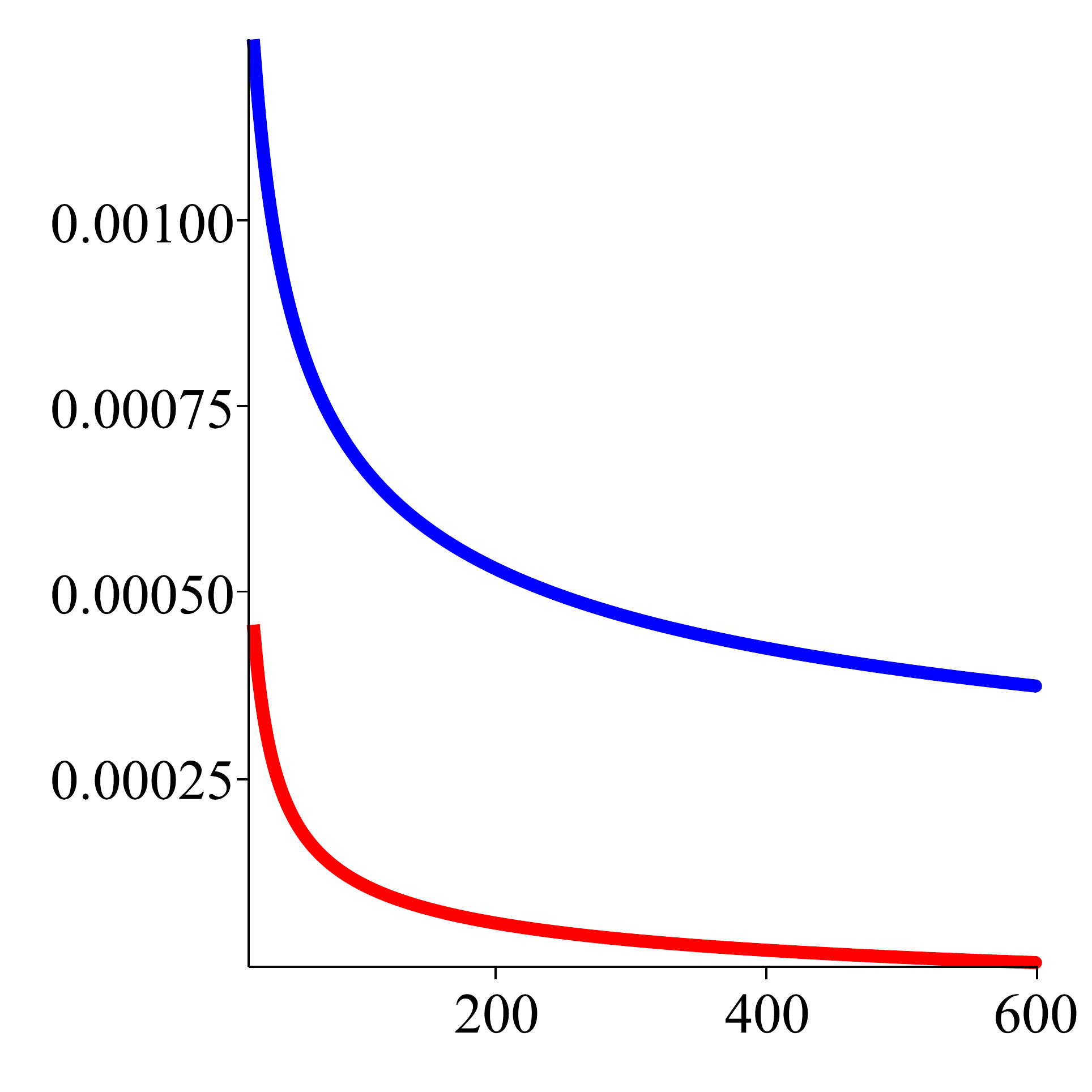} 
	\end{tabular}	
	\caption{$\Delta_{n,M}^{(c)}$ (in red) vs $\Delta_{n,M}^{(v)}$ 
	(in blue): $20\le n\le 200$ and $M=0,1,2,3,4$ (in left to right 
	order).}\label{F:hs}
	\end{center}
\end{figure}
From Figure~\ref{F:hs}, we see that while \eqref{E:Jn} is numerically 
better than its circular counterpart \eqref{E:In} (or $M=0$. as 
already observed in \cite{Harris1967}), more terms in the asymptotic 
expansions show that both expansions are indeed comparable, and their 
numerical performance depends on the number of terms used. 

We also observed a very similar pattern (as Figure~\ref{F:hs}) for 
Bell numbers when $\phi(z) = e^z-1$; see 
\cite[\href{https://oeis.org/A000110}{A000110}]{OEIS2022}.
\begin{figure}[!ht]
	\begin{center}
	\begin{tabular}{ccccc}
		\includegraphics[height=2.5cm]{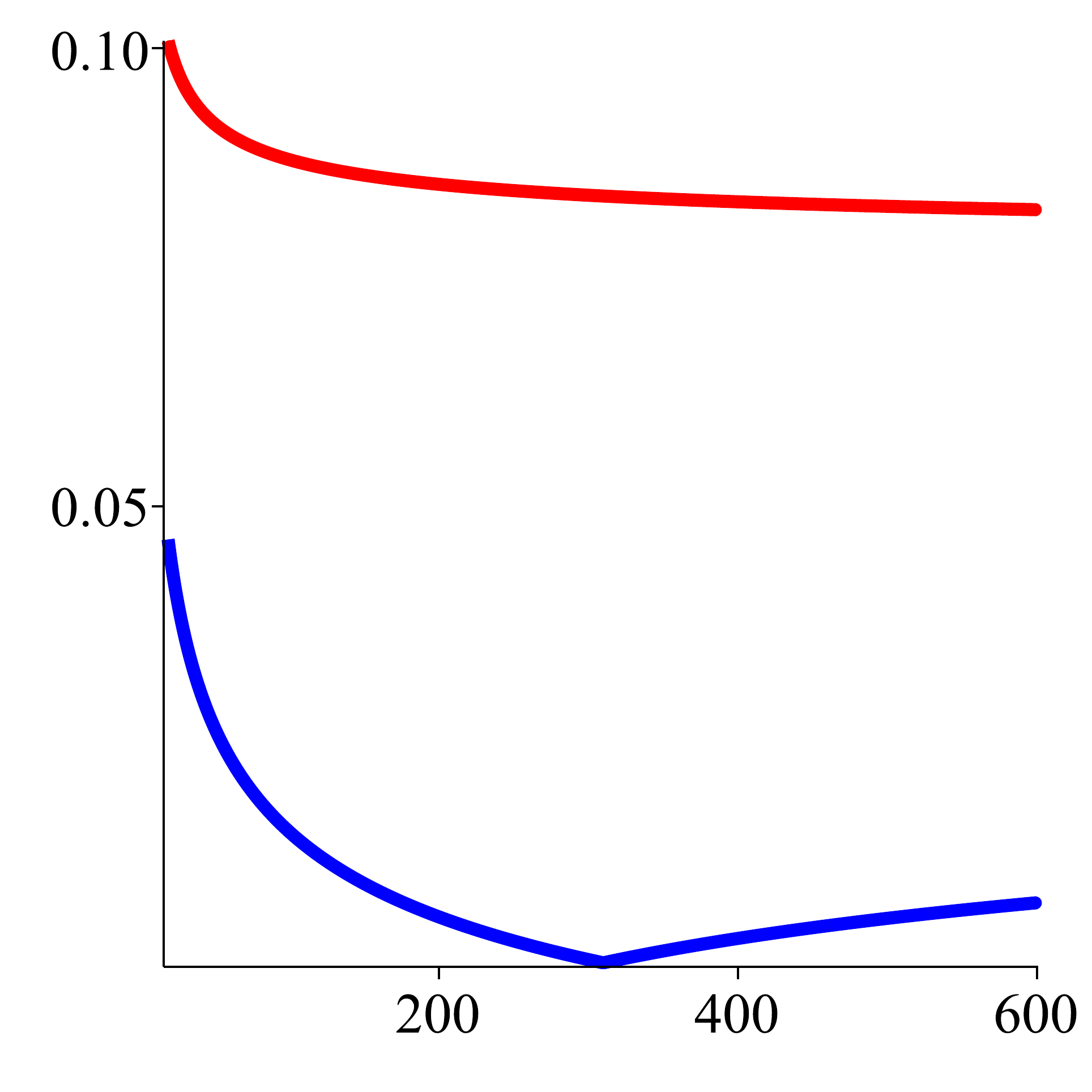}
		& \includegraphics[height=2.5cm]{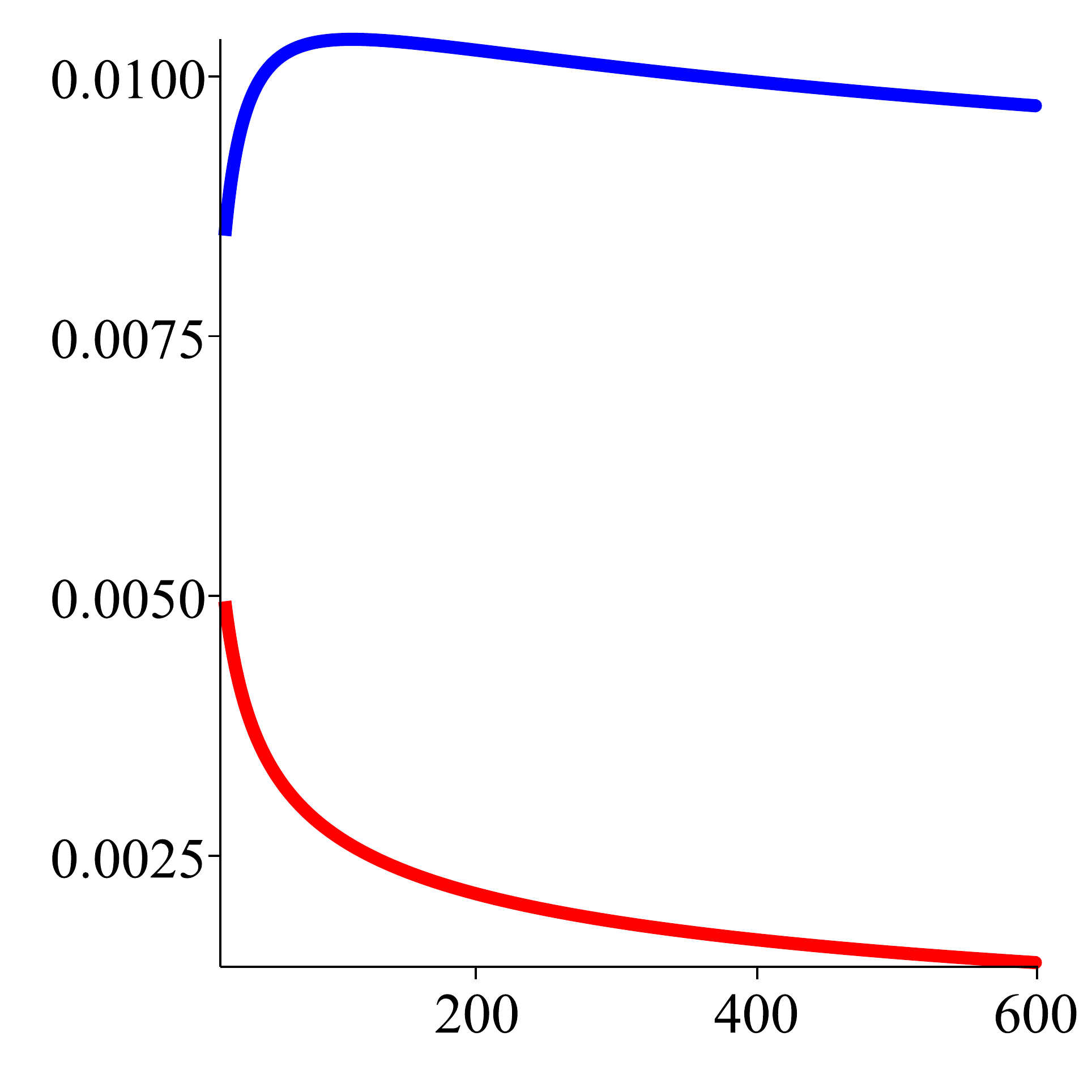}
		& \includegraphics[height=2.5cm]{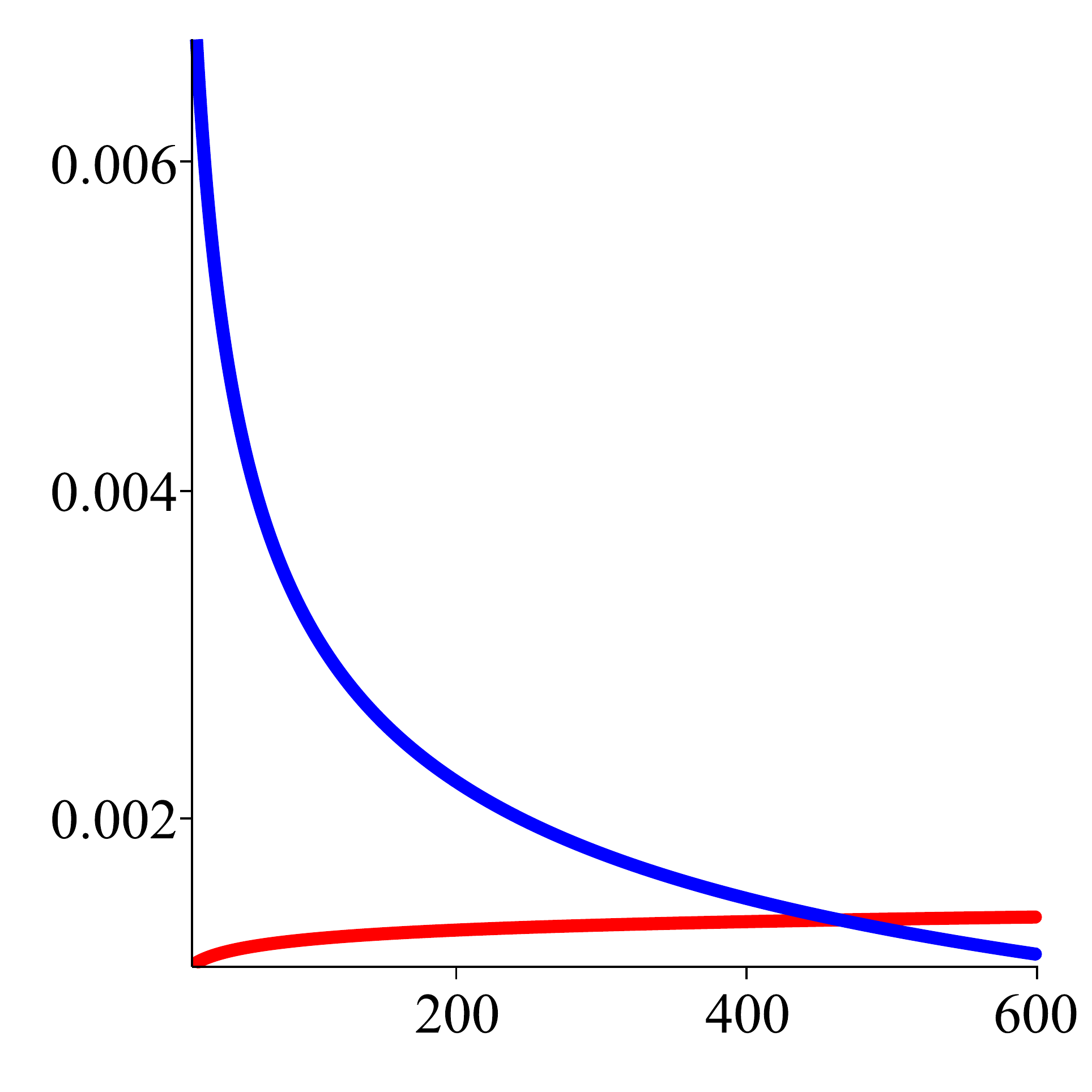}
		& \includegraphics[height=2.5cm]{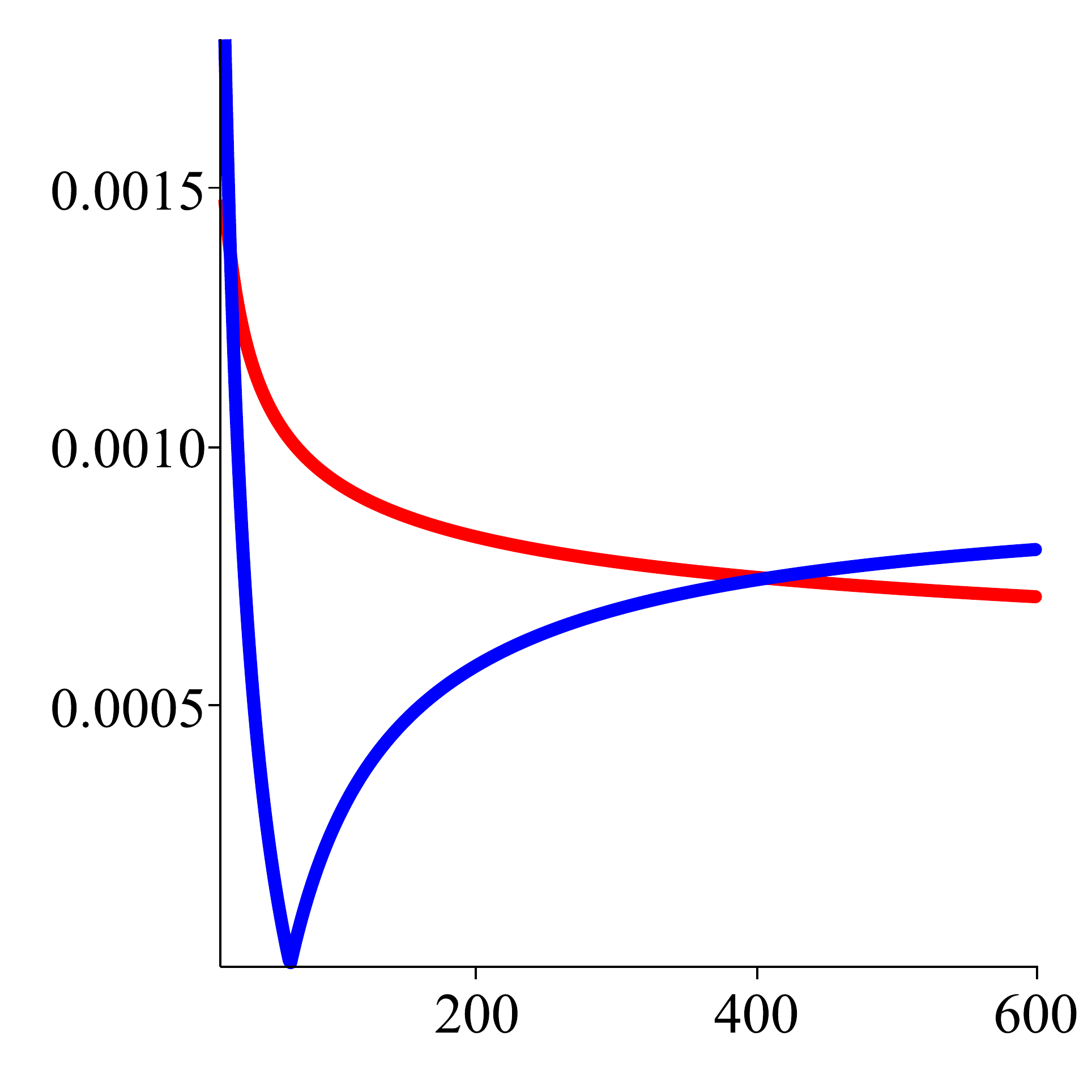}
		& \includegraphics[height=2.5cm]{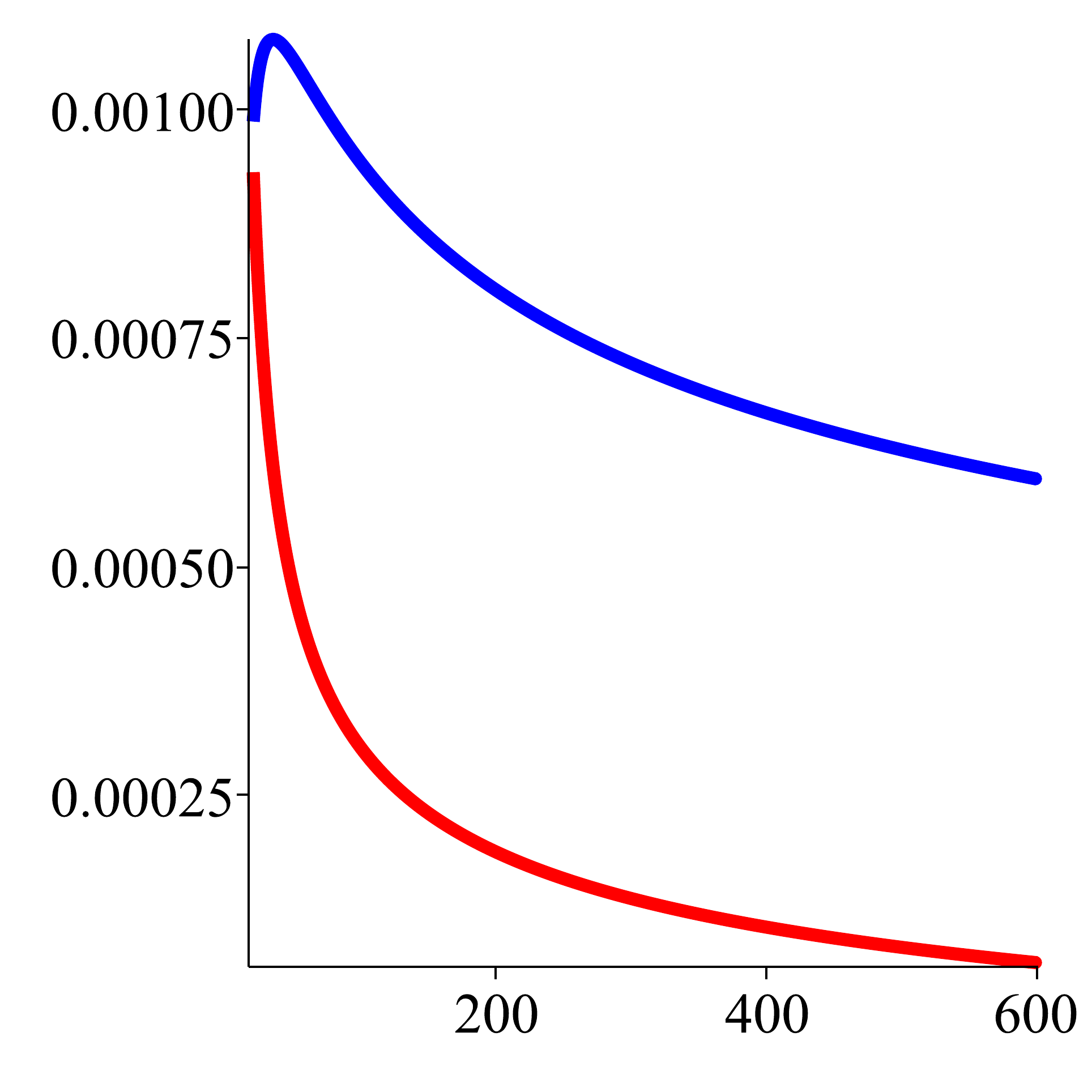} 
	\end{tabular}	
	\caption{$\Delta_{n,M}^{(c)}$ (in red) vs $\Delta_{n,M}^{(v)}$ 
	(in blue) in the case of Bell numbers: $15\le n\le 200$ and 
	$M=0,1,2,3,4$ (in left to right order).}
	\end{center}
\end{figure}

%
%
%
%

In the case of $\phi(z) = \frac{z}{1-z}$, $a_n := n![z^n]e^{z/(1-z)}$ 
enumerates the number of partitions of $\{1,\dots,n\}$ into any 
number of ordered subsets; see 
\cite[\href{https://oeis.org/A000262}{A000262}]{OEIS2022}. 
Justification of an asymptotic expansion is straightforward and 
similar to integer partition problems.
\begin{figure}[!ht]
	\begin{center}
	\begin{tabular}{ccccc}
		\includegraphics[height=2.5cm]{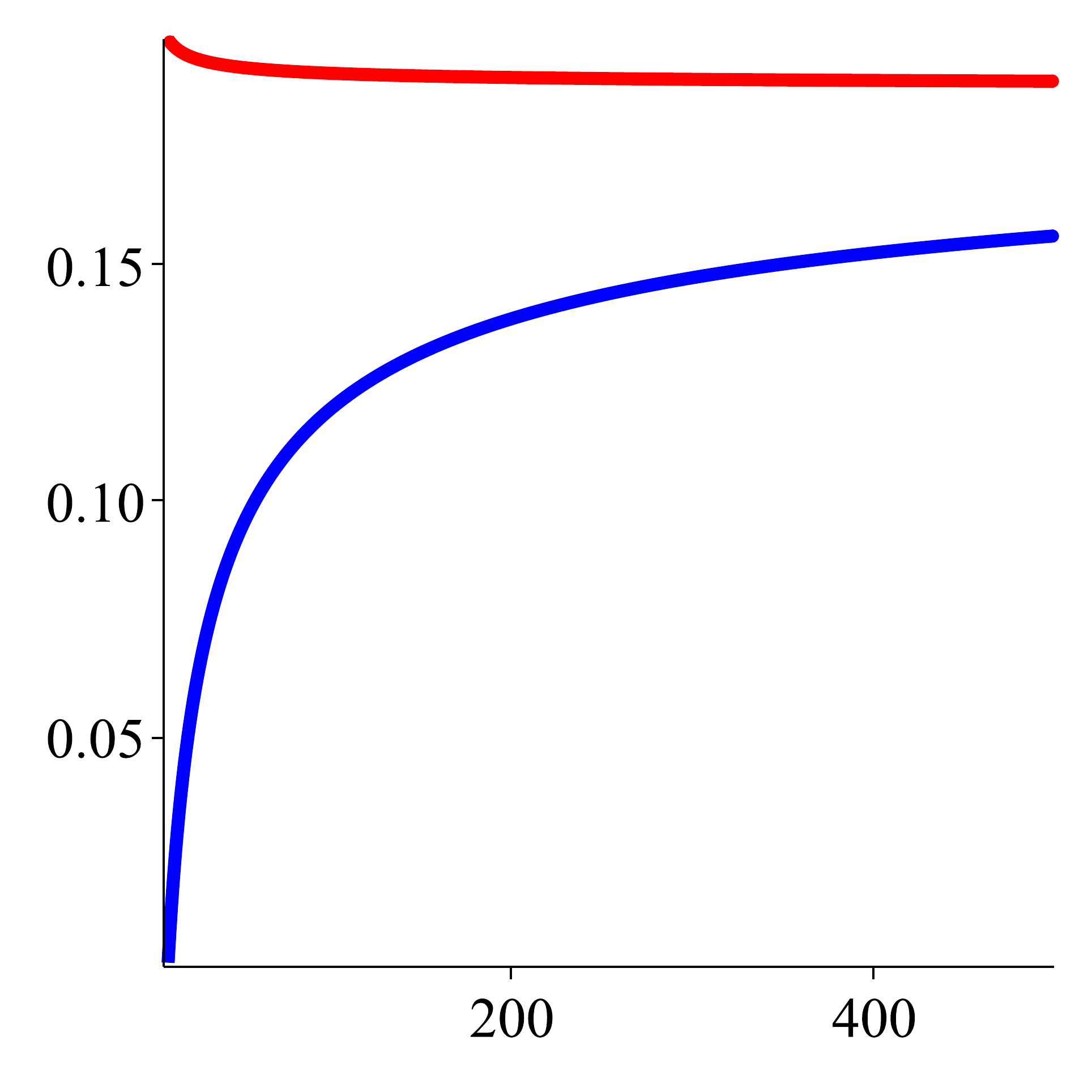}
		& \includegraphics[height=2.5cm]{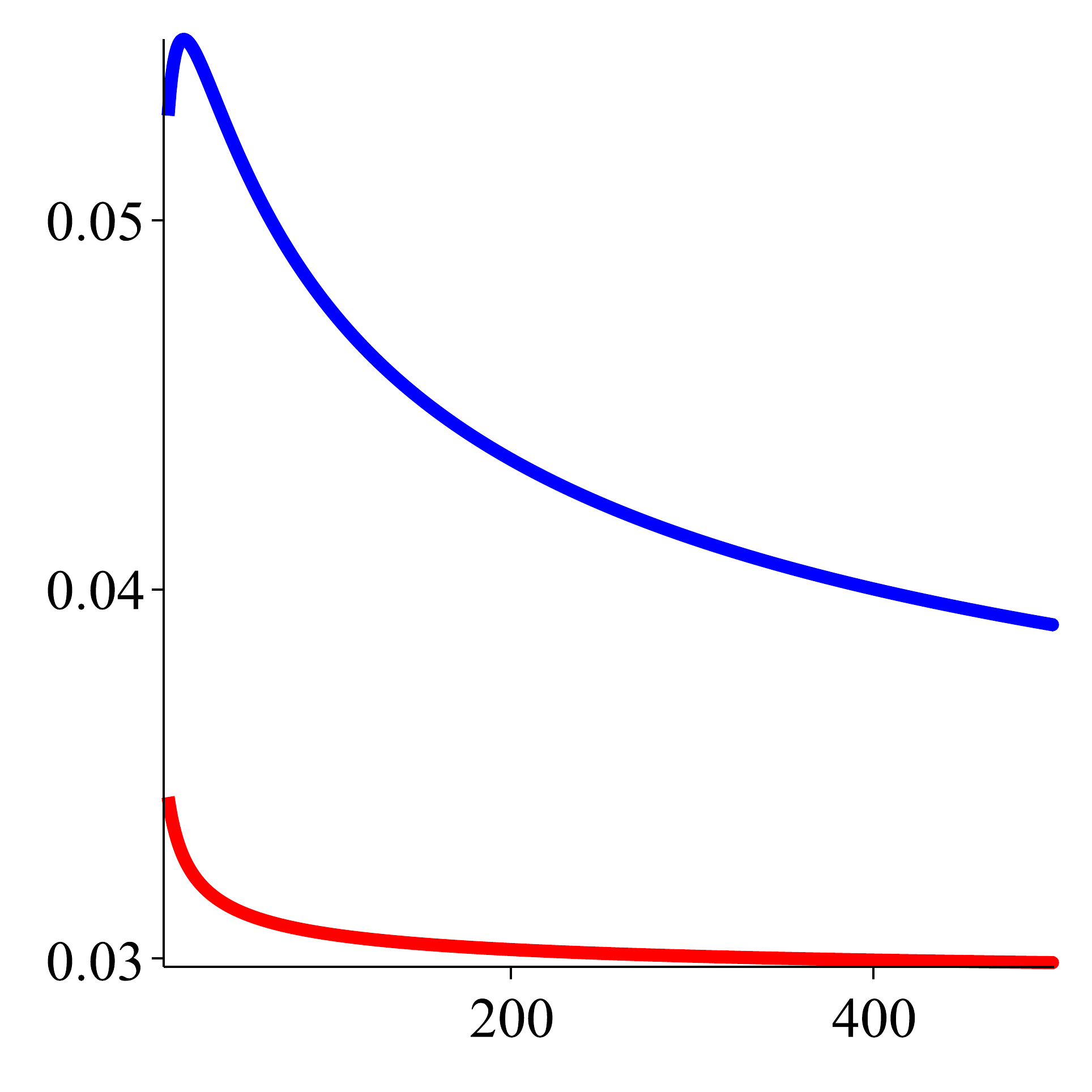}
		& \includegraphics[height=2.5cm]{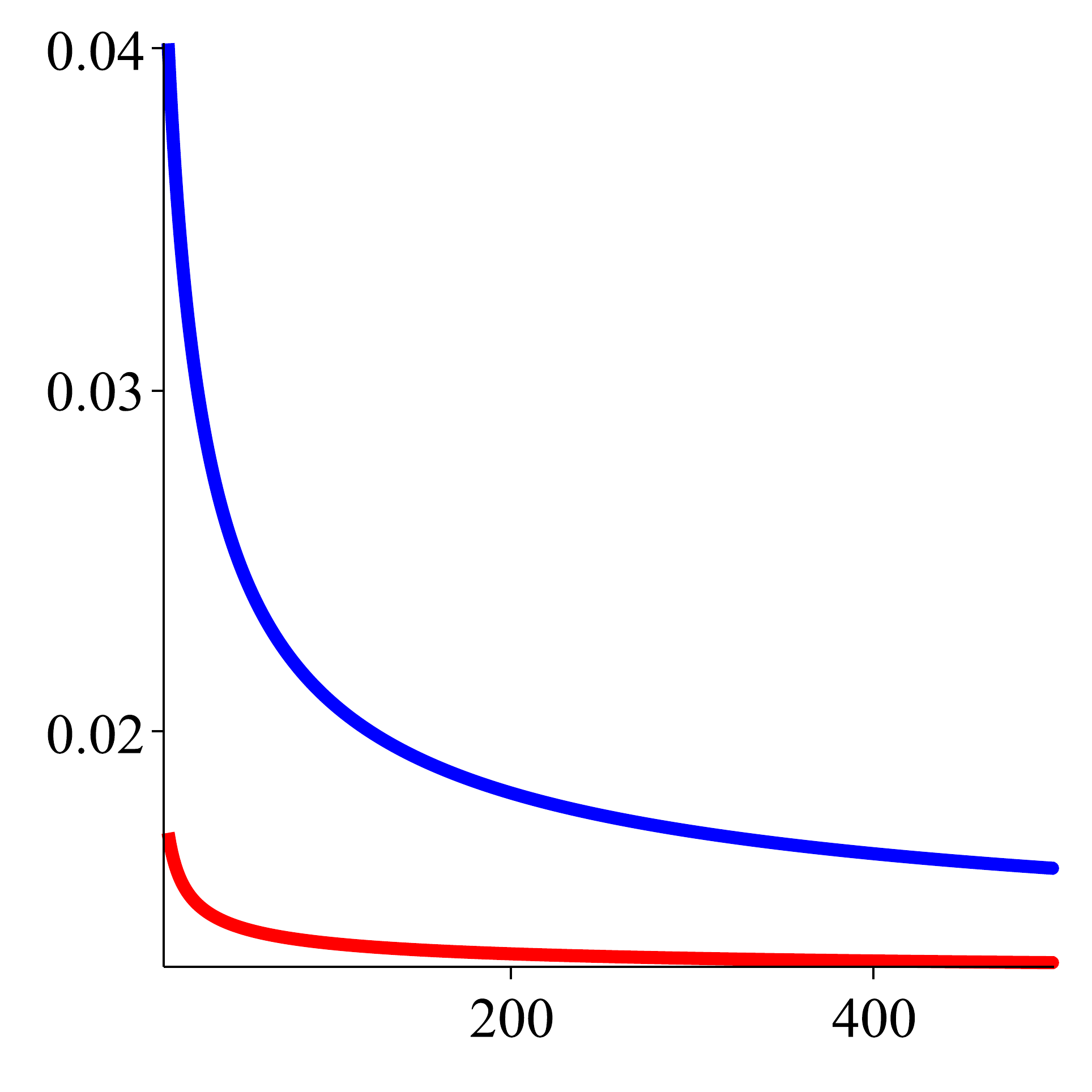}
		& \includegraphics[height=2.5cm]{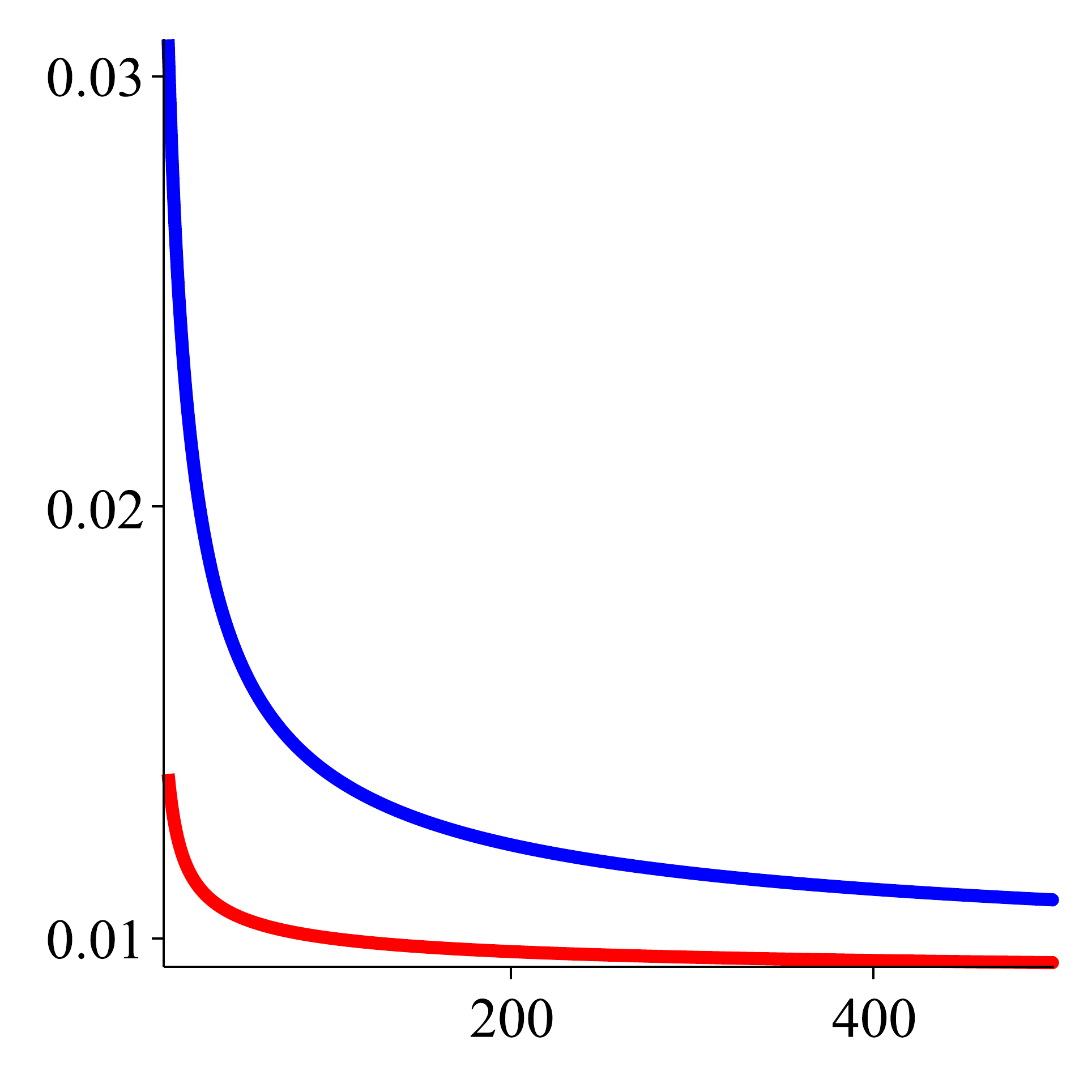}
		& \includegraphics[height=2.5cm]{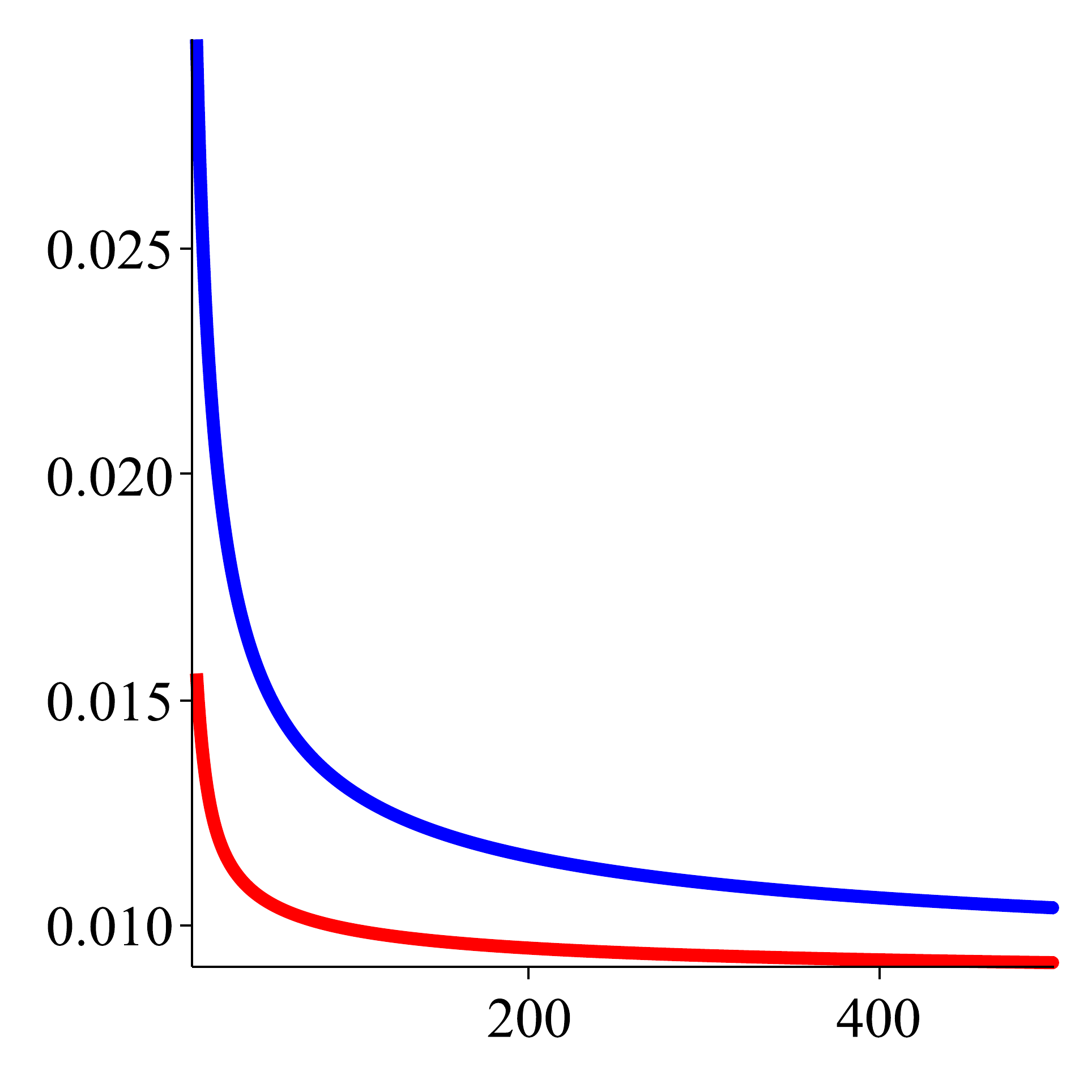} 
	\end{tabular}	
	\caption{$\Delta_{n,M}^{(c)}$ (in red) vs $\Delta_{n,M}^{(v)}$ 
	(in blue) in the case of $\phi(z)=\frac{z}{1-z}$: $10\le n\le 
	200$ and $M=0,1,2,3,4$ (in left to right order). Here 
	$\Delta_{n,M}^{(\cdot)}$ is defined as in \eqref{E:Delta} but 
	with $\frac{n^{M+1}}{(\log n)^{M+1}}$ there replaced by 
	$n^{\frac12(M+1)}$.}
	\end{center}
\end{figure}
One sees that the circular version is numerically better except for 
$M=0$.

Finally, consider the case $\phi(z) = z+\frac12z^2$, whose 
coefficients (times $n!$) enumerate the number of self-inverse 
permutations on $n$ elements; see 
\cite[\href{https://oeis.org/A000085}{A000085}]{OEIS2022}. Since all 
coefficients of $\phi(z)$ are positive, an asymptotic expansion by 
saddle-point method is possible by known results of Moser and Wyman 
in the 1950s \cite{Moser1957}; see also \cite{Odlyzko1995}. In this 
case, we plot the difference $\Delta_{n,M}^{(c)}-\Delta_{n,M}^{(v)}$ 
because the two curves are too close to be distinguishable.
\begin{figure}[!ht]
	\begin{center}
	\begin{tabular}{ccccc}
		\includegraphics[height=2.8cm]{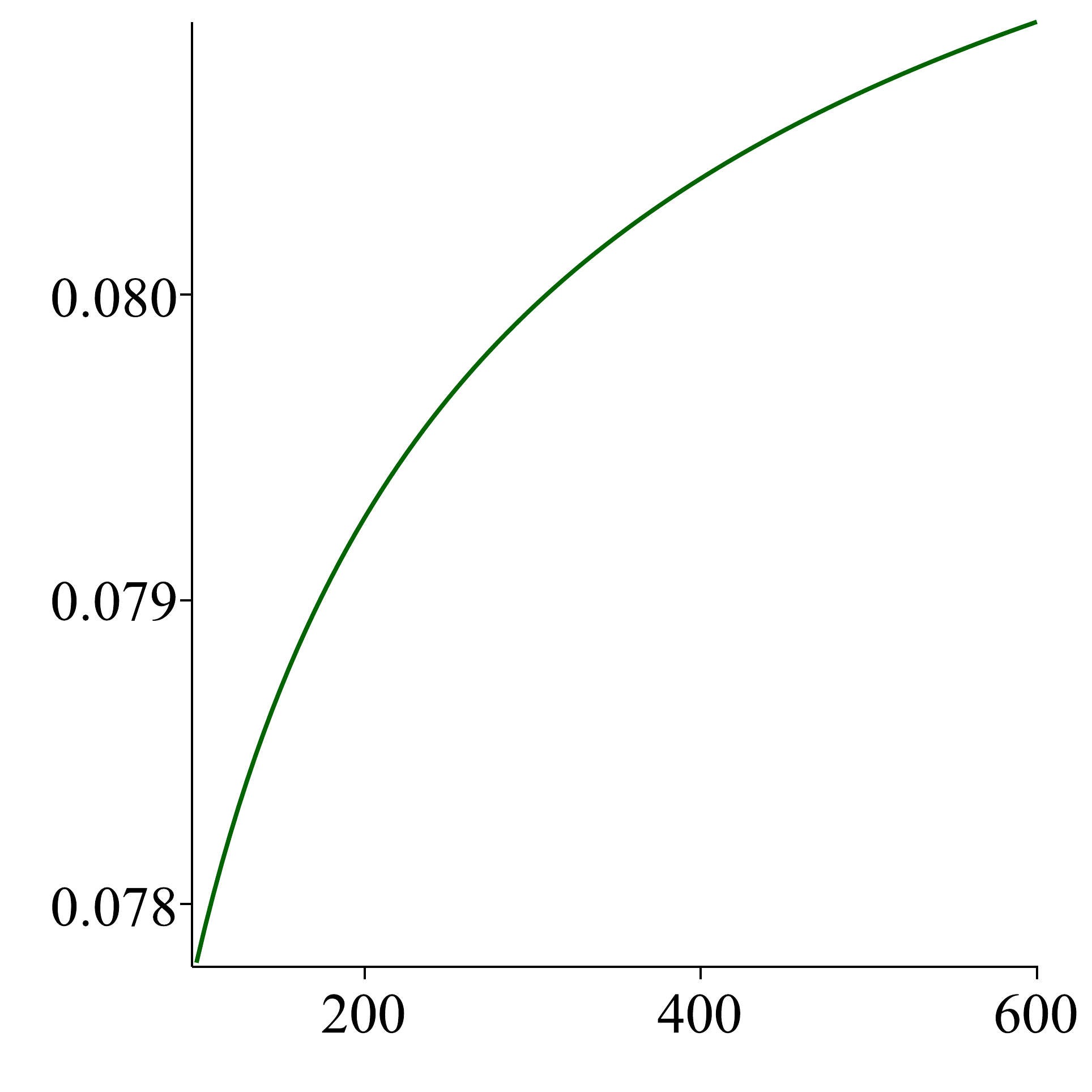}
		& \includegraphics[height=2.8cm]{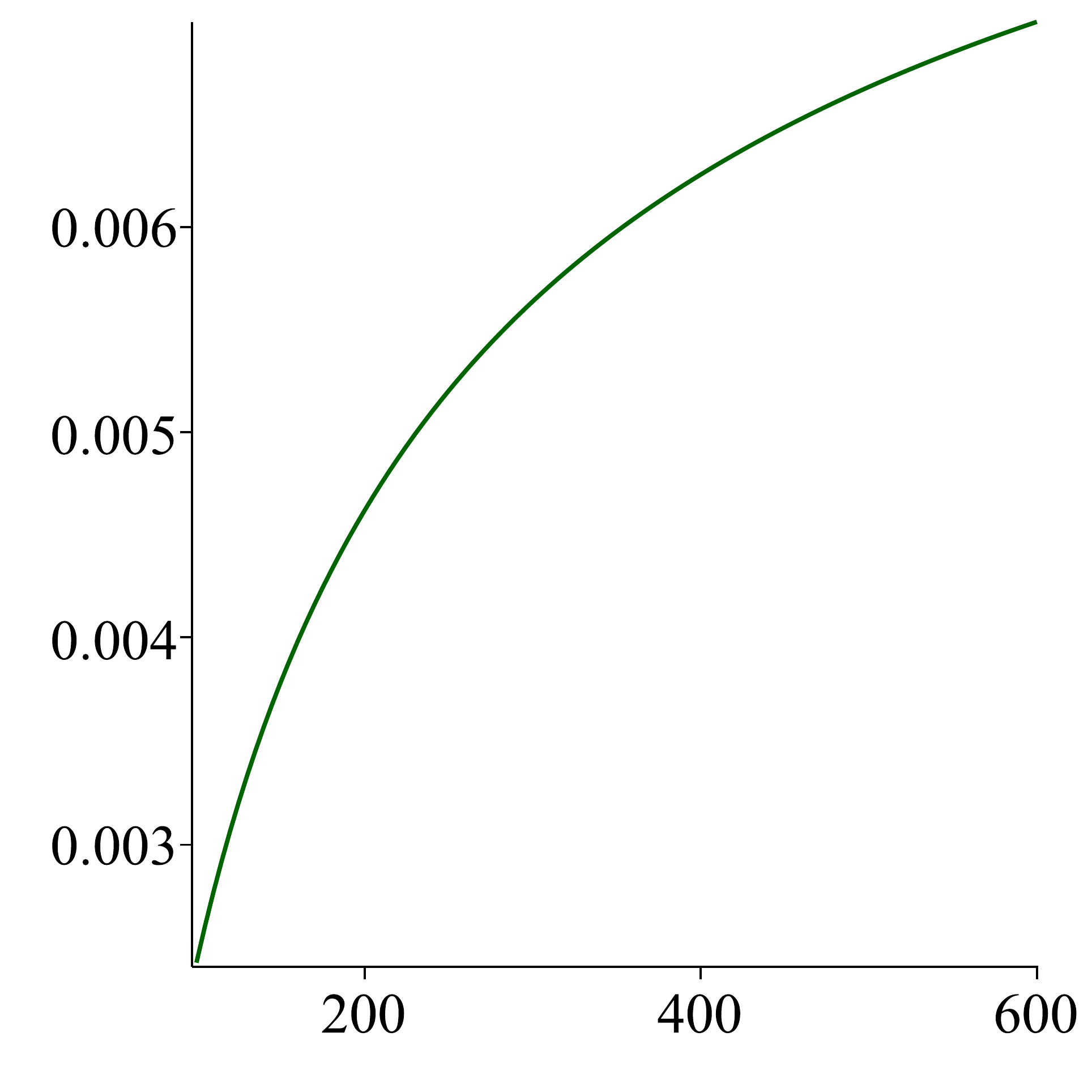}
		& \includegraphics[height=2.8cm]{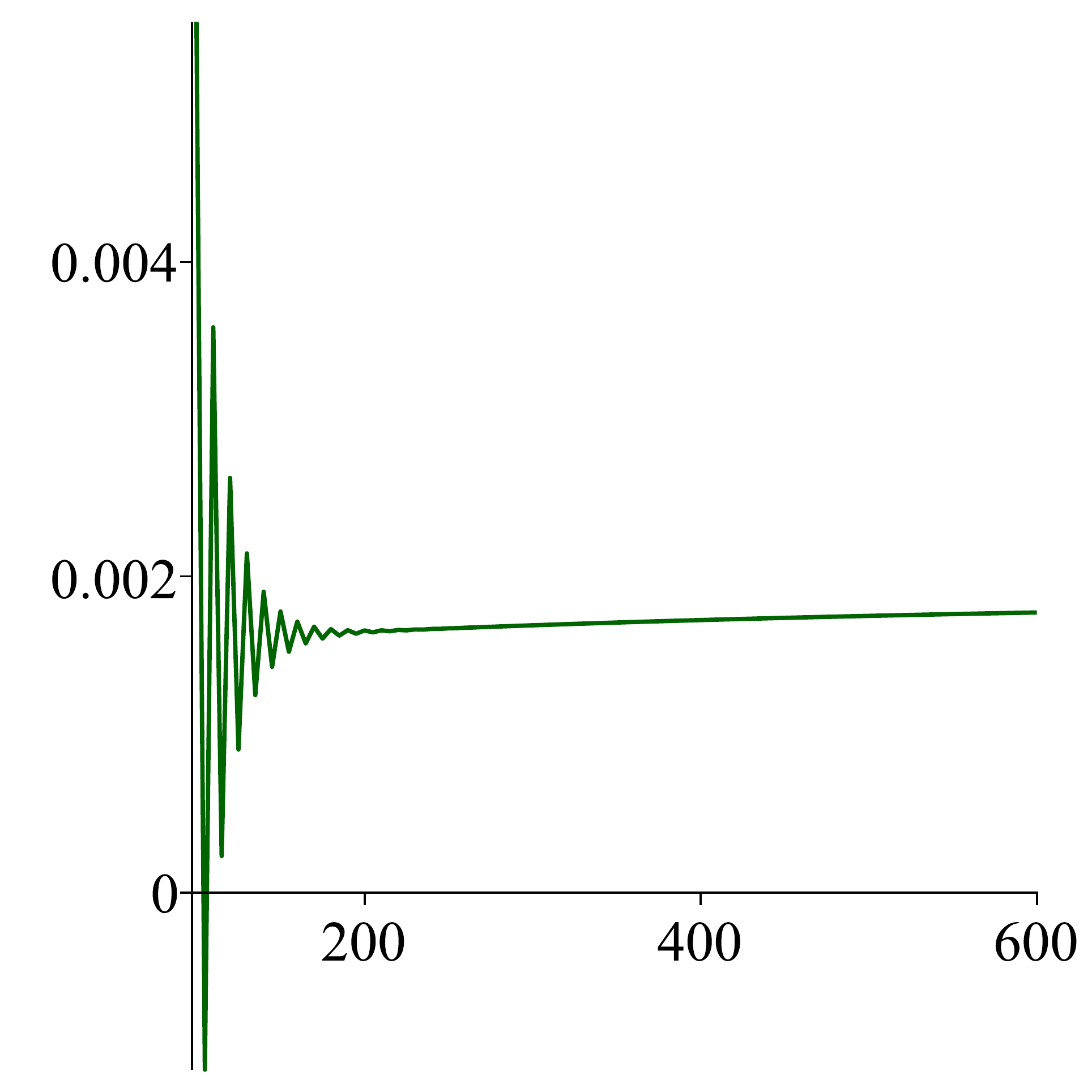}
		& \includegraphics[height=2.8cm]{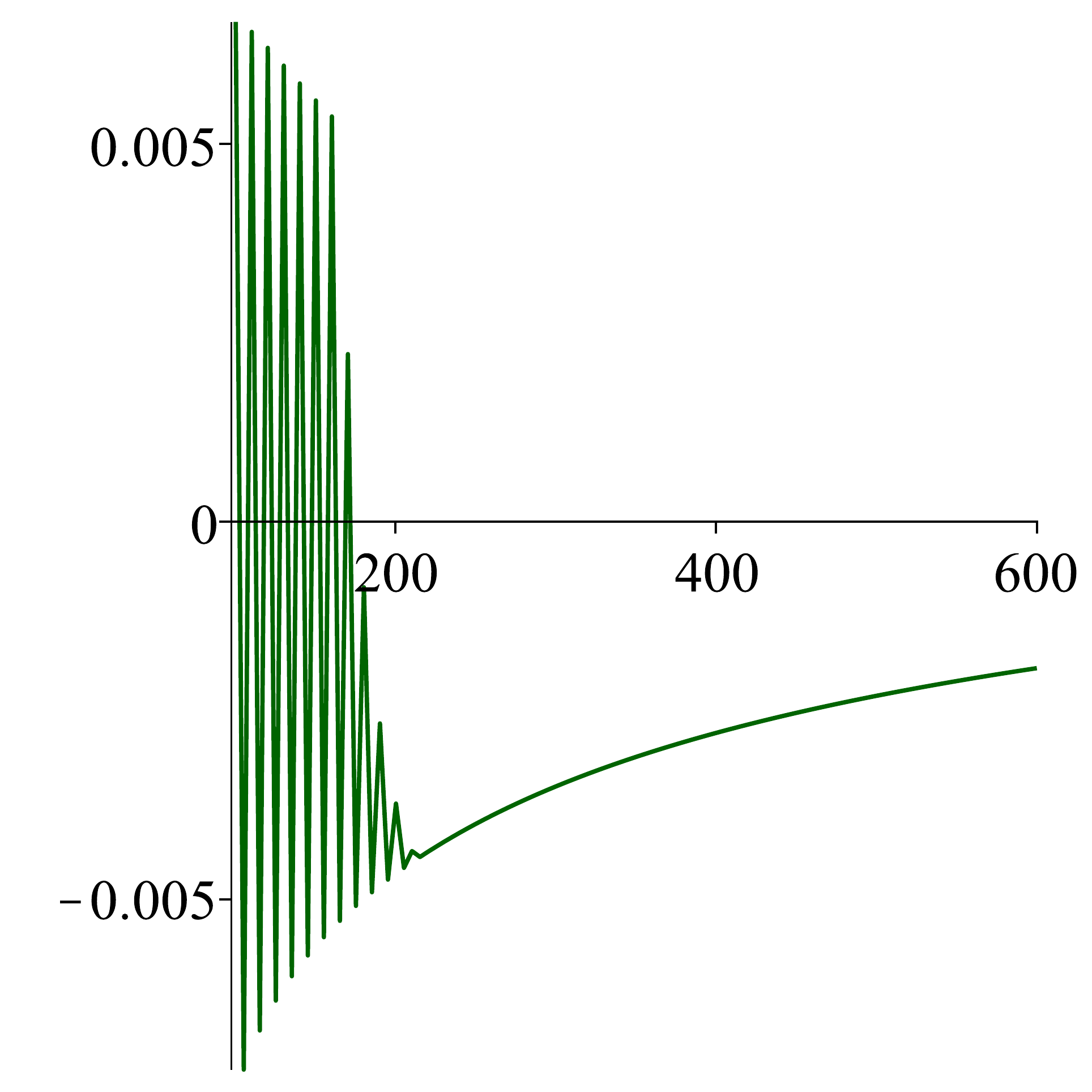}
	\end{tabular}	
	\caption{$\Delta_{n,M}^{(c)}-\Delta_{n,M}^{(v)}$ 
	in the case of $\phi(z)=z+\frac12z^2$: $100\le n\le 
	200$ (with step 5) and $M=0,1,2,3$ (in left to right order). 
	Here $\Delta_{n,M}^{(\cdot)}$ is defined as in \eqref{E:Delta} 
	but with $\frac{n^{M+1}}{(\log n)^{M+1}}$ there replaced by 
	$n^{M+1}$.}
	\end{center}
\end{figure}

In summary, we clarified here the often unclear situation of which
version of the saddle-point contour to choose, and provided concrete
examples for a more detailed comparison, from both analytic and
numerical viewpoints. Such a clarification will be of instructional
value, in addition to its own methodological interests.

\section{A Lagrangean framework}
\label{S:lag}

Consider now the Lagrangean form
\[
	[z^n]f(z)\eqtext{with}
	f=zG(f),
\]
where $G(0)>0$. By Lagrange inversion relation, the Taylor 
coefficients satisfy 
\begin{align}\label{E:lag-inv}
	n[z^n]f(z) = [t^{n-1}]G(t)^n\qquad(n\ge1).
\end{align}
This is one of the rare classes of functions for which both the 
singularity analysis and the saddle-point method apply well 
(see \cite[p.~590]{Flajolet2009} and \cite{Hwang2018}) because of 
\eqref{E:lag-inv}. Under the following sub-criticality conditions:
\begin{equation}\label{E:sub-critic}
	\left\{
	\begin{split}
		&\bullet 
		\text{$G$ is analytic in $|z|<\rho$, $0<\rho<\infty$;}\\
		&\bullet
		\text{$[z^j]G(z)\ge0$ and $\gcd\{j\,:\,[z^j]G(z)>0\}=1$;}\\
		&\bullet
		\text{the equation $zG'(z)=G(z)$ has a unique positive 
		solution $\rho_0\in(0,\rho)$}
	\end{split}\right.
\end{equation}
%
it is proved in \cite{Hwang2018} via singularity analysis that
\[
	[z^n]f(z) \sim \sum_{k\ge0}c_k\binom{n-k-\frac32}{n},
	\eqtext{with} c_k = \frac{(-1)^k}{k}[t^{k-1}]
	\llpa{\frac{1-\frac{(\rho+t)G(\rho)}
	{\rho G(\rho+t)}}{t^2}}^{-\frac12k},
\]
where $\rho := \frac{r}{G(r)}$ with $r>0$ solving the equation 
$rG'(r)=G(r)$.

Here we examine this framework from the saddle-point method
viewpoint. It turns out that the two asymptotic expressions we
obtained above via two different contours are the same in this
framework, and they are related to each other by a direct change of
variables.
\begin{thm} Write $\phi(z) = \log G(z)$. \
Under the subcriticality conditions \eqref{E:sub-critic},
\begin{align}\label{E:lag2}
	n[z^n]f(z)
	\sim \frac{R^{1-n}G(R)^n}
	{\sqrt{2\pi n}\,\sigma(R)}
	\sum_{m\ge0}h_{2m}\frac{(-1)^m(2m)!}
    {2^mm!}\, (\sigma(R)^2n)^{-m},
\end{align}
where $R>0$ solves the equation $R\phi'(R)=1$, 
$\sigma(R)^2 = R\phi'(R)+R^2\phi''(R)$ and 
\begin{align}\label{E:lag-hm}
	h_m = [v^m]\llpa{\frac{\frac12\sigma(R)^2v^2}
	{\phi(R(1+v))-\phi(R)-R\phi'(R)
    \log(1+v)}}^{\frac12(m+1)}.
\end{align}
\end{thm}
The expression for the coefficients in the expansion \eqref{E:lag2} 
is much simpler than that given in \cite[Theorem 2]{Hwang2018}. 
\begin{proof}
We work out the asymptotic expansion in the circular case, the 
vertical line case then following from a change of variables. As an 
asymptotic expansion of the form \eqref{E:lag2} can either be 
justified by singularity analysis as in \cite{Hwang2018} or by 
standard saddle-point analysis as in \cite{Drmota1994}, we focus here 
on the (formal) calculation of the coefficients. By \eqref{E:lag-inv}
\begin{align*}
	n[z^n]f(z) &= \frac1{2\pi i}\oint_{|z|=r}
	z^{-n}G(z)^n\dd z\\
	&= \frac{r^{1-n}}{2\pi i}
	\int_{-\pi i}^{\pi i} e^u(e^{-u} G(re^u))^n\dd u\\
	&\thickapprox \frac{r^{1-n}G(r)^{n}}{2\pi\sigma(r)}
	\int_{-\ve i}^{\ve i} e^{\frac12nv^2}g(v) \dd v,
\end{align*}
where $\sigma(r)^2 := r\phi'(r)+r^2\phi''(r)$, $r\phi'(r)=1$, $g(v) 
= e^u \frac{\dd u}{\dd v} = \frac{\dd{}}{\dd v}e^u$, and 
\[
	\frac{\phi(re^u)-\phi(r)-r\phi'(r)u}
	{\sigma(r)^2} = \frac{v^2}2.
\]
We then deduce that 
\begin{align}\label{E:lag1}
	n[z^n]f(z)
	\sim \frac{r^{1-n}G(r)^n}
	{\sqrt{2\pi n}\,\sigma(r)}
	\sum_{m\ge0}g_{2m}\frac{(-1)^m(2m)!}
    {2^mm!}\, (\sigma(r)^2n)^{-m},
\end{align}
where
\begin{align}\label{E:lag-gm}
	g_m = [t^m]e^t\llpa{\frac{\frac12\sigma(r)^2t^2}
	{\phi(re^t)-\phi(r)-r\phi'(r)t}}^{\frac12(m+1)}.
\end{align}
By the change of variables $v=e^t-1$, we obtain the expression 
\eqref{E:lag-hm}, which can also be obtained directly by beginning 
with the coefficient integral with the change of variables $z=R(1+v)$.
\end{proof}

In particular, if $\phi(z)=z$ or $G(z)=e^z$, then 
\[
 	n[z^n]f(z) = \frac{n^{n-1}}{(n-1)!}
 	= [z^{n-1}]e^{nz},
\]
and we obtain the same expressions as derived above for Stirling's 
formula. 

\subsection{Catalan numbers}
For simplicity, we consider only Catalan numbers for which 
$G(z) = (1-z)^{-1}$ or $\phi(z) = -\log(1-z)$, so that
\[
	[z^n]f(z) = [z^n]\frac{1-\sqrt{1-4z}}2
	= \frac1n\binom{2n-2}{n-1}.
\]
Then the positive solution of the equation $r\phi'(r)=1$ is given by 
$r=\frac12$, and from either of the two equations \eqref{E:lag-hm} and
\eqref{E:lag-hm}, we have the asymptotic expansion ($\sigma(R)^2=2$)
\[
	\frac1n\binom{2n-2}{n-1}
	\sim \frac{4^{n-1}}{\sqrt{\pi}}\,
	\sum_{m\ge0}
	h_{2m}\frac{(-1)^m(2m)!}{4^mm!}\, n^{-m-\frac32},
\]
and the identity 
\[
	h_m := [y^m]\llpa{\frac{y^2}
	{-\log(1-y^2)}}^{\frac12(m+1)}
	= [v^m]e^v\llpa{\frac{v^2}
	{-\log(2-e^v)-v}}^{\frac12(m+1)},
\]
for $m\ge0$, which follows simply by the change of variables 
$y=e^v-1$. Note particularly that $h_{2l+1}=0$ for $l\ge0$.

On the other hand, by singularity analysis (see \cite{Flajolet2009})
\begin{align*}
	\frac1n\binom{2n-2}{n-1}
	&= [z^n]\frac{1-\sqrt{1-4z}}{2}
	= -\frac{4^n}{4\pi i}
	\int e^{nt}\sqrt{1-e^{-t}}\dd t\\
	&\sim -\frac{4^n}{4\pi i}\sum_{m\ge0}b_m
	\int_{\mathcal{H}} e^{nt}t^{m+\frac12}\dd t
	\sim -\frac{4^n}{2}\sum_{m\ge0}
	\frac{b_m}{\Gamma(-m-\frac12)}\,n^{-m-\frac32},
\end{align*}
where $b_m = [t^m]\lpa{(1-e^{-t})/t}^{\frac12}$. Now by the relation 
\begin{align*}
	-\frac1{\Gamma(-m-\frac12)}
	= \frac{(-1)^m(2m+2)!}
	{\sqrt{\pi}(m+1)!4^{m+1}},
\end{align*}
we then get
\begin{align}\label{E:cata-ae}
	\frac1n\binom{2n-2}{n-1}
	\sim \frac{4^n}{2\sqrt{\pi}}
	\sum_{m\ge0}\frac{b_m(-1)^m(2m+2)!}
	{(m+1)!4^{m+1}}\, n^{-m-\frac32}.
\end{align}
It follows that $h_{2m} = (2m+1)b_m$, which can also be proved 
directly by a change of variables. 

For large $m$, it is known (see \cite[p.~39]{Norlund1961}) that 
\[
	b_m \sim \frac{\sin(\frac12m\pi)}
	{\sqrt{\pi}}\, (2\pi)^{-n}m^{-\frac32},
\]
implying that the expansion \eqref{E:cata-ae} is divergent for
$n\ge1$. Since the right-hand side is zero when $m$ is even, we can
refine the approximation by the same singularity analysis and obtain
\[
	b_m \sim (-1)^{\lfloor{\frac12m}\rfloor}(2\pi)^{-m}
	\times \begin{cases}
		\frac{3\sqrt{\pi}}{4}\,m^{-\frac52}, 
		&\text{if $m$ is even};\\
		\frac{1}{\sqrt{\pi}}\, m^{-\frac32}, 
		&\text{if $m$ is odd}.
	\end{cases}
\]

On the other hand, we can improve the asymptotic expansion by noting 
that 
\[
	\llpa{\frac{1-e^{-t}}t}^{\frac12}
	= e^{-\frac14t}\Lpa{\frac 2t\sinh \frac t2}^{\frac12}
	= e^{-\frac14t}\sum_{m\ge0}b_{2m}' t^{2m};
\]
thus, by the same singularity analysis
\[
	\frac1n\binom{2n-2}{n-1}
	\sim\frac{4^n}{2\sqrt{\pi}}
	\sum_{m\ge0} \frac{b_{2m}'(4m+2)!}
	{(2m+1)!4^{2m+1}}
    \, \lpa{n-\tfrac14}^{-2m-\frac32},
\]
an expansion containing only even terms.

Yet another way to derive an asymptotic expansion for Catalan 
numbers is as follows. Let $G(z)= (1+z)^2$. Then 
\[
	\frac1{n+1}\binom{2n}{n}
	= \frac1n[t^{n-1}](1+z)^{2n},
\]
and we have
\[
	\frac1{n+1}\binom{2n}{n}
	\sim \frac{4^n}{\sqrt{\pi}}
	\sum_{m\ge0}h_{2m}\frac{(-1)^m(2m)!}{m!}\, n^{-m-\frac32},
\]
where
\[
	h_m := [v^m]\llpa{\frac{v^2}
	{4\log\frac{(1+\frac12v)^2}{1+v}}}^{\frac12{(m+1)}}.
\]
By a direct change of variables, we have
\[
	h_{2m}
	= 4^{-m}[y^m](2e^y-1)
	\sqrt{\frac{y}{1-e^{-y}}}.
\]

\bibliographystyle{abbrv}
\bibliography{stirling-variants} 
\end{document}